\documentclass[11pt]{amsart}
\usepackage{amsmath, amsthm, amssymb, amsfonts, verbatim}
\usepackage[pdftex]{graphicx}
\usepackage[bookmarks]{hyperref}
\reversemarginpar            

\newtheoremstyle{modthm}
     {15pt}
     {3ptpt}
     {\itshape}
     {}
     {\bfseries}
     {.}
     {.5em}
     {}

\newtheoremstyle{modrem}
     {15pt}
     {0pt}
     {\rmfamily}
     {}
     {\itshape}
     {.}
     {.5em}
     {}

\theoremstyle{modthm}
\newtheorem{thm}{Theorem}
\newtheorem{prop}{Proposition}[section]
\newtheorem{lem}{Lemma}[section]

\newtheorem*{acknowledgment}{Acknowledgments}
\newtheorem*{theorem*}{Theorem}
\newtheorem{cor}{Corollary}
\theoremstyle{modrem}
\newtheorem{rem}{Remark}

\def\be{\begin{equation}}
\def\ee{\end{equation}}
\newcommand{\Div}{\mathrm{div}}
\newcommand{\R}{\mathbb{R}}

\begin{document}
\title[Critical $N$-dimensional Burgers' Equation]{Regularity of solutions for the critical $N$-dimensional Burgers' equation. }
\author[Chan]{Chi Hin Chan}
\address{Department of Mathematics, The University of Texas at Austin}
\email{cchan@math.utexas.edu}

\author[Czubak]{Magdalena Czubak}
\address{Department of Mathematics, University of Toronto}
\email{czubak@math.toronto.edu}

\begin{abstract}  
We consider the fractional Burgers' equation on $\R^N$ with the critical dissipation term.  We follow the parabolic De-Giorgi's method of Caffarelli and Vasseur \cite{Driftdiffusion} and show existence of smooth solutions given any initial datum in $L^2(\R^N)$.
\end{abstract}
\maketitle
\maketitle
\section{Introduction}
In this paper we investigate the regularity of the solutions to the critical $N$-dimensional Burgers' equation.  The equation is given by
\begin{equation}\label{Burger}
\partial_{t}\theta + \sum_{j=1}^{N} \theta \cdot \partial_{j} \theta = - (-\triangle)^{\frac{1}{2}} \theta,
\end{equation}
where $\theta: [0,\infty) \times \R^N \rightarrow \R$.  \eqref{Burger} is called critical, because of the invariance with respect to the scaling transformation given by
\[
\theta_\lambda(t,x)=\theta(\lambda t,\lambda x),
\]
and is a special case of
\[
\partial_{t}\theta + \sum_{j=1}^{N} \theta \cdot \partial_{j} \theta = - (-\triangle)^{\alpha} \theta, 
\]
where $0<\alpha<1$.  In a recent paper Kiselev, Nazarov and Shterenberg \cite{Kiselevmain} have done an extensive study for the $1$-dimensional Burgers' equation in the periodic setting, which covers the subcritical case $\frac{1}{2} < \alpha < 1$, the critical case $\alpha = \frac{1}{2}$, and also the supercritical case $0< \alpha < \frac{1}{2}$.  Among the results obtained in \cite{Kiselevmain}, the authors prove the global in time existence of locally H\"older continuous solutions for the critical case $\alpha = \frac{1}{2}$ with respect to periodic initial datum $\theta_{0} \in L^{p}(\mathbb{R})$ with $1<p<\infty$.  On the other hand,  Dong, Du and Li \cite{DongDuLi} also cover all the values of $\alpha$, but both with and without the periodic setting, and with the emphasis on the finite time blow up in the supercritical case.  In another recent work, Miao and Wu \cite{Miao} establish global well-posedness of the critical Burgers' equation in critical Besov spaces $B_{p,1}^{1/p}(\mathbb{R})$.  For further background and motivation for the fractional Burgers' equation we refer our readers to \cite{Kiselevmain},\cite{Miao},\cite{DongDuLi}.\newline\indent
The main goal of this paper is to establish the following theorem
\begin{thm}\label{mainthm}
Given any initial datum $u_{0} \in L^{2}(\mathbb{R}^{N})$ there exists a global weak solution $\in L^{\infty}(0,\infty ; L^{2}(\mathbb{R}^{N})) \cap L^{2} (0,\infty ; \dot H^{\frac{1}{2}}(\mathbb{R}^{N}))$ of the critical Burgers' equation \eqref{Burger} such that
\begin{itemize}
\item $\theta (0,\cdot ) = \theta_{0} $ in the $L^{2}(\mathbb{R}^{N})$-sense.
\item For every $t>0$, we have $\theta(t,\cdot) \in L^{\infty}(\mathbb{R}^{N})$.
\item $\theta$ is locally H\"older continuous.
\end{itemize}
\end{thm}
\begin{cor}\label{cor1}
The solution $\theta$ obtained in Theorem \ref{mainthm} is smooth.
\end{cor}
As far as we know this is the first result for the regularity of the $N$-dimensional Burgers' equation.  However, we are also aware of the $1$-dimensional regularity results in the periodic setting due to Kiselev, Nazarov and Shterenberg in \cite{Kiselevmain}, but \emph{we} do not know whether their method of modulus of continuity can be generalized to the $N$-dimensional setting.  In addition, we would like to emphasize that the method of our proof relies completely on the methods of Caffarelli and Vasseur \cite{Driftdiffusion}.  In \cite{Driftdiffusion} authors develop a very delicate parabolic De-Giorgi's method, which leads them to the global smooth solutions for the critical quasi-geostrophic equation in the $N$-dimensional setting.   We add here that Kiselev, Nazarov, and Volberg \cite{Kiselevquasi} also obtain the same existence result as \cite{Driftdiffusion} in $2$-dimensional setting by using the method of modulus of continuity.  Moreover the method of modulus of continuity is also employed in \cite{Kiselevmain}, \cite{Miao}.\\
\indent
Before we explain the way in which our paper originates from \cite{Driftdiffusion}, we briefly remark on the regularity problem of solutions for the quasi-geostrophic equation, which is a question parallel to the regularity problem of solutions for the Burgers' equation.  Since the finding of the global weak solutions by Resnick in his thesis \cite{Resnickthesis}, there has been a significant amount of work devoted to addressing the existence and uniqueness of smooth solutions for the quasi-geostrophic equation (see for example, \cite{Constantinone},\cite{Constantintwo},\cite{Dongquasi},\cite{DongandNatasa}).  Of course, we have to mention that the existence of global smooth solutions for the critical quasi-geostrophic equation with respect to initial datum $\in L^{2}(\mathbb{R}^{N})$ has recently been established independently by Caffarelli and Vasseur~\cite{Driftdiffusion}, and Kiselev, Nazarov and Volberg~\cite{Kiselevquasi}.\newline\indent
We are now ready to clarify the relationship between our paper and the work of Caffarelli and Vasseur \cite{Driftdiffusion}.  As we have mentioned, the purpose of this paper is to perform suitable modifications on the parabolic De-Giorgi's method developed in \cite{Driftdiffusion}, so that, after our modifications, such a parabolic De-Giorgi's method will give the existence of locally H\"older continuous solutions to the critical $N$-dimensional Burgers' equation with respect to the initial datum $\theta_{0} \in L^{2}(\mathbb{R}^{N})$.  
We would like to bring to the readers' attention the following main issue.\newline\indent
In \cite{Driftdiffusion} the authors study the following critical $N$-dimensional quasi-geostrophic equation
\begin{gather*} 
\partial_{t} \theta + v\cdot \nabla \theta = - (-\triangle )^{\frac{1}{2}} \theta ,\\
\Div v = 0 ,\\
v_{j} = R_{j} [\theta ] ,
\end{gather*}
where $\theta : (0, \infty )\times \mathbb{R}^{N} \rightarrow \mathbb{R}$ is a scalar valued solution and $v$ is the velocity field related to $\theta$ by some selected singular integral operators $R_{j}$. Besides characterizing the fractional Laplacian $(-\triangle )^{\frac{1}{2}}$ via harmonic extension of functions to the upper half plane (see \cite{CaffarelliSilvestre} for more on the harmonic extension), one of the key stepping stones in \cite{Driftdiffusion} is the following local energy inequality appearing in section 3 of \cite{Driftdiffusion}.
\begin{prop}
(Caffarelli and Vasseur \cite{Driftdiffusion}) Let $\theta : (0, T )\times \mathbb{R}^{N} \rightarrow \mathbb{R}$ be a weak solution for the critical $N$-dimensional quasi-geostrophic equation for which the respective velocity field $v : (0, T )\times \mathbb{R}^{N} \rightarrow \mathbb{R}^{N}$ verifies
\begin{equation*}
\|v\|_{L^{\infty}(0,T ; BMO(\mathbb{R}^{N}))} + \sup_{t\in [0,T]} \Big\vert\int_{B_{2}} v(t,x) dx\Big\vert \leqslant C_{v}.
\end{equation*}
Then it follows that $\theta$ satisfies the following local energy inequality for some universal constant $\Phi_{v}$ depending only on $C_{v}$
\begin{equation}\label{QuasiLEI}
\begin{split}
&\frac{1}{2}\int_{\sigma}^{t} \int_{B_{2}^{*}} |\nabla (\eta \theta_{+}^{*})|^{2} dx dz ds
+ \frac{1}{2} \int_{B_{2}} (\eta \theta_{+})^{2}(t,x) dx\\
&\quad \leqslant \frac{1}{2} \int_{B_{2}} (\eta \theta_{+})^{2}(\sigma ,x) dx 
 + \Phi_{v} \int_{\sigma}^{t}\int_{B_{2}}|\nabla \eta |^{2} \theta_{+}^{2}
+ \int_{\sigma}^{t}\int_{B_{2}^{*}} |\nabla \eta |^{2} (\theta_{+}^{*})^{2} dx dz dt,
\end{split}
\end{equation}
where $(\sigma,t)\in (0,T)$, $B_{2} = \{x \in \mathbb{R}^{N} : |x| < 2\}$, $B_{2}^{*} = B_{2}\times [0,2]$, $\theta^{*}$ is the harmonic extension of $\theta$ to $\mathbb{R}^{N}\times [0, \infty )$, $\theta_{+}^{*} = \theta^{*} \chi_{\{  \theta^{*} > 0 \}},$ $\theta_{+} = \theta \chi_{\{  \theta > 0 \}},$ and $\eta$ is some cut off function supported in $B_{2}^{*}$.
\end{prop}
In order to use the above local energy inequality \eqref{QuasiLEI} freely Caffarelli and Vasseur make the key observation that: \emph{if $\theta$ is a solution of the critical quasi-geostrophic equation, then any other function $u = \beta \{ \theta - L\}$, with arbitrary constants $\beta > 0$ and $L \in \mathbb{R}$, gives another solution of the same quasi-geostrophic equation}.  Such an observation is of crucial importance since this allows the authors to use the above local energy inequality \emph{with the same universal constant $\Phi_{v}$} for any functions in the form of $u = \beta \{ \theta - L   \} $ (that is, not just for the solution $\theta$ itself).  This provides a lot of advantage whenever it is necessary to shift the focus from the solution $\theta$ to some appropriate $u = \beta \{ \theta - L   \}$.  
Unfortunately, in the case of the critical Burgers' equation such a key observation is no longer valid.  This is the main obstacle (and actually the only one) we are facing in borrowing the parabolic De-Giorgi's method from \cite{Driftdiffusion}.  However, we can overcome this difficulty by making the following important observation: \emph{after the local energy inequality \eqref{QuasiLEI} was established in \cite{Driftdiffusion}, the authors actually relied only on the local energy inequality \eqref{QuasiLEI}, rather then the critical quasi-geostrophic equation itself}.\newline\indent
Because of this observation, when we are dealing with a solution $\theta$ of the critical Burgers' equation, we are motivated to focus on the more general function $u = \beta \{ \theta - L   \}$, with constants $|\beta | >0$ and $L\in \mathbb{R}$, and we try to obtain the corresponding local energy inequality satisfied by  $u = \beta \{ \theta - L   \}$.  Indeed, we will find that: if $\theta$ solves the $N$-dimensional critical Burger's equation, then, $u = \beta \{ \theta - L   \} $ will satisfy the following local energy inequality
\begin{equation}\label{CompareLEI}
\begin{split}
&\frac{1}{2}\int_{\sigma}^{t} \int_{B_{2}^{*}} |\nabla (\eta u_{+}^{*})|^{2} dx dz ds
+ \frac{1}{2} \int_{B_{2}} (\eta u_{+})^{2}(t,x) dx\\
&\quad \leqslant \frac{1}{2} \int_{B_{2}} (\eta u_{+})^{2}(\sigma ,x) dx
 + 2N C_{N} [|L| + \|\theta \|_{L^{\infty}([-4,0]\times \mathbb{R}^{3})} ]^{2}
\int_{\sigma}^{t}\int_{B_{2}}|\nabla \eta |^{2} u_{+}^{2}\\
&\quad\quad+ \int_{\sigma}^{t}\int_{B_{2}^{*}} |\nabla \eta |^{2} (u_{+}^{*})^{2} dx dz dt,
\end{split}
\end{equation}
where $u^{*}$ is the harmonic extension of $u$, $u_{+}^{*} = u^{*} \chi_{\{  u^{*} > 0 \}}$, $u_{+} = u \chi_{\{  u > 0 \}}$, and $\eta$ is some cut off function supported in $B_{2}^{*}$.\newline\indent
Now, let us compare inequalities \eqref{QuasiLEI} and \eqref{CompareLEI}.  In the case of the critical Burgers' equation, the constant $ 2N C_{N} [|L| + \|\theta \|_{L^{\infty}([-4,0]\times \mathbb{R}^{3})} ]^{2}$ plays the same role as the universal constant $\Phi_{v}$ appearing in \eqref{QuasiLEI}. However, the universal constant $\Phi_{v}$ in \eqref{QuasiLEI} \emph{remains unchanged} while we replace the solution $\theta$ by $\beta \{ \theta - L   \}$.  In contrast, inequality \eqref{CompareLEI} does not enjoy this stability property, since the quantity $ 2N C_{N} [|L| + \|\theta \|_{L^{\infty}([-4,0]\times \mathbb{R}^{3})} ]^{2} $ might become large compared with $ \|\theta \|_{L^{\infty}([-4,0]\times \mathbb{R}^{3})} $ when the shifting-level $L$ is changing. Because of this, we have to make sure that the constant $ 2N C_{N} [|L| + \|\theta \|_{L^{\infty}([-4,0]\times \mathbb{R}^{3})} ]^{2} $ is under control by a certain integer multiple of $ \lVert\theta \rVert_{L^{\infty}([-4,0]\times \mathbb{R}^{3})} $ at any time we need to employ \eqref{CompareLEI} in our paper.\newline\indent
In fact, once we succeed in applying inequality \eqref{CompareLEI} to $\beta \{\theta - L\}$ the main obstacle we are facing disappears and the parabolic De Giorgi's method as developed in \cite{Driftdiffusion} leads to the proof of Theorem 1.\newline\indent
The set up of the paper is as follows.  In section \ref{weaksolutions} we show existence of the $L^\infty$ bounded weak solution.  Section \ref{leisection} is devoted to the proof of the local energy inequality.  In section \ref{mainproof} we establish some fundamental lemmas, which when combined together with Theorem 2 (see below) result in the proof of Theorem 1.  In section \ref{higher} we discuss how to extend the H\"older continuity to higher regularity.
\begin{acknowledgment}
Both authors are extremely grateful to Professor Alexis Vasseur for suggesting the problem and for his guidance and support throughout the process.  Without him this work would not have come into existence.
\end{acknowledgment}


\section{Existence of $L^{\infty}$-bounded weak solutions.}\label{weaksolutions}
To prove the existence of H\"older continuous solutions for the $N$-dimensional critical Burgers' equation \eqref{Burger} it is necessary for us to establish the existence of $L^{\infty}$-bounded solutions first.  To that end, we provide a proof for the following theorem
\begin{thm}\label{weaksolutiontheorem}
For any given initial datum $\theta_{0} \in L^{2}(\mathbb{R}^{N})$, there exists a weak solution $\theta \in L^{\infty}(0,\infty ; L^{2}(\mathbb{R}^{N})) \cap L^{2} (0,\infty ; \dot H^{\frac{1}{2}}(\mathbb{R}^{N}))$ of the critical Burgers' equation \eqref{Burger} which satisfies the following two properties
\begin{itemize}
\item $\theta (0,\cdot ) = \theta_{0} $ in the $L^{2}(\mathbb{R}^{N})$-sense.
\item For every $t>0$, we have $ \|\theta (t,\cdot )\|_{L^{\infty}(\mathbb{R}^{N})} \leqslant \frac{C_{N}}{t^{\frac{N}{2}}} \|\theta_{0}\|_{L^{2}(\mathbb{R}^{N})} $, where $C_{N}$ is some universal constant depending only on $N$. 
\end{itemize}
\end{thm}
\begin{proof}
We start by considering the following \emph{modified critical Burgers' equation},
\begin{equation}\label{Artificalequation}
\partial_{t}\theta + \sum_{j=1}^{N} \psi_{R}(\theta ) \cdot \partial_{j} \theta = -(-\triangle)^{\frac{1}{2}} \theta  + \epsilon \triangle \theta. 
\end{equation}
In the above, an artificial diffusion term $\epsilon \triangle \theta$ is included, and the nonlinear term $\theta \cdot \partial_{j} \theta$ is now replaced by $\psi_{R}(\theta ) \cdot \partial_{j} \theta$, where $R>1$ is an arbitrarily chosen quantity, and $\psi_{R} : \mathbb{R} \rightarrow \mathbb{R}$ is the continuous piecewise linear function given by 
\begin{equation*}
\psi_{R}(\lambda ) = \lambda\cdot\chi_{\{  -R < \lambda < R    \}} + R\cdot  \chi_{\{  |\lambda | \geqslant R    \}}.
\end{equation*}
Due to the addition of the artificial diffusion term $\epsilon \triangle \theta$, it is not hard to convince ourselves that the existence of (Leray-Hopf) weak solutions for the above modified Burgers' equation \eqref{Artificalequation} can easily be established through an application of the standard Galerkin approximation.  Because of this, for the rest of this proof we freely employ the weak solutions of the modified Burgers' equation \eqref{Artificalequation}.\newline\indent
Now, given an initial datum $\theta_{0} \in L^{2}(\mathbb{R}^{N})$, we consider a weak solution  $\theta$ of \eqref{Artificalequation} in the Leray-Hopf class $L^{\infty}(0,\infty ; L^{2}(\mathbb{R}^{N})) \cap L^{2} (0,\infty ; \dot H^{1}(\mathbb{R}^{N}))$ satisfying $\theta (0,\cdot ) = \theta_{0} $ in the $L^{2}(\mathbb{R}^{N})$-sense.  We will employ the standard De-Giorgi's method to prove that $\theta$ is $L^{\infty }$-bounded over $[t_{0} ,\infty )\times \mathbb{R}^{N}$, for every $t_{0} > 0$.  Before this can be done, it is necessary to show that our solution $\theta$ of \eqref{Artificalequation} verifies the following vanishing property for almost every $t\in (0,\infty )$, and at every truncation level $L > 0$
\begin{equation}\label{vanishing}
\int_{\mathbb{R}^{N}}  \psi_{R}(\theta ) \cdot\partial_{j}\theta\cdot\{\theta - L \}_{+} dx = 0,
\end{equation}
where $\{\theta-L\}_{+} = \{\theta-L\} \chi_{\{\theta > L \}}$.  For the sake of convenience, we write $\theta_{L} = \{\theta - L \}_{+}$.  We then observe
\begin{equation*}
\psi_{R}(\theta ) \cdot \partial_{j} \theta \cdot \theta_{L} 
= \frac{1}{2} \{   \partial_{j} [ \psi_{R}(\theta ) \theta_{L}^{2} ]  -  \theta_{L}^{2} \partial_{j}[ \psi_{R}(\theta ) ]                            \}
= \frac{1}{2} \{   \partial_{j} [ \psi_{R}(\theta ) \theta_{L}^{2} ]  -  \theta_{L}^{2} \chi_{\{ -R < \theta < R    \}} \partial_{j} \theta                           \}.
\end{equation*}
By taking the integral over $\mathbb{R}^{N}$ of the above identity, we yield
\begin{equation*}
\int_{\mathbb{R}^{N}}  \psi_{R}(\theta ) \cdot \partial_{j} \theta \cdot \theta_{L} dx
=-\frac{1}{2} \left\{   \int_{\mathbb{R}^{N}}   \theta_{L}^{2} \chi_{\{ -R < \theta < 0    \}} \partial_{j} \theta dx    +  \int_{\mathbb{R}^{N}} \theta_{L}^{2} \chi_{\{ 0 < \theta < R    \}} \partial_{j} \theta  dx       \right\},
\end{equation*}
so we will succeed in justifying \eqref{vanishing}, if we can show 
\begin{equation*}
\int_{\mathbb{R}^{N}}   \theta_{L}^{2} \chi_{\{ -R < \theta < 0    \}} \partial_{j} \theta dx   = \int_{\mathbb{R}^{N}} \theta_{L}^{2} \chi_{\{ 0 < \theta < R    \}} \partial_{j} \theta  dx = 0 .
\end{equation*}
Start with the second term.  Without the loss of generality\footnote{This is because $ \theta_{L}^{2} \chi_{\{ 0 < \theta < R    \}} \partial_{j} \theta     $ automatically vanishes if $L > R $.} we assume that $0 \leqslant L \leqslant R$.  We note
\begin{equation*}
\theta_{L}^{2} \chi_{\{ 0 < \theta < R    \}} \partial_{j} \theta  = \theta_{L}^{2} \chi_{\{ \theta_{L} < R-L     \}} \partial_{j} \theta_{L}=-\theta_{L}^{2}\partial_j[R-L - \theta_{L} ]_{+}.
\end{equation*}
Then a computation shows
\begin{equation}\label{vanishingtwo}
\begin{split}
\theta_{L}^{2} \chi_{\{ 0 < \theta < R    \}} \partial_{j} \theta
& = - (R-L)^{2} \partial_{j}[R-L - \theta_{L} ]_{+}
  + (R-L) \partial_{j}[(R-L - \theta_{L})_{+}^{2}]\\
&\quad\quad  - \frac{1}{3} \partial_{j} [(R-L - \theta_{L})_{+}^{3}].
\end{split}
\end{equation}
Next
\begin{equation*}
[R-L - \theta_{L} ]_{+} = (R-L) - \{ \theta_{L} \chi_{\{ \theta_{L} < R-L \}}   
+ (R-L) \chi_{\{ \theta_{L} \geqslant R - L   \}}   \}.
\end{equation*}
Observe
\[
\theta_{L} \chi_{\{ \theta_{L} < R-L \}} \in  L^{1}(\mathbb{R}^{N})\quad\mbox{and}\quad (R-L) \chi_{\{ \theta_{L} \geqslant R - L   \}} \in L^{1}(\mathbb{R}^{N}),
\]
and since these functions are also in $L^\infty(\R^N)$, they are in $L^p(\R^N), 1<p<\infty$, so
it follows from \eqref{vanishingtwo} that we must have 
\[
\int_{\mathbb{R}^{N}} \theta_{L}^{2} \chi_{\{ 0 < \theta < R    \}} \partial_{j} \theta  dx = 0.
\] 
In exactly same way, we can also show that
\[
\int_{\mathbb{R}^{N}} \theta_{L}^{2} \chi_{\{ -R < \theta < 0    \}} \partial_{j} \theta  dx = 0. 
\]
Hence the validity of property \eqref{vanishing} is established.\\

We are now ready to apply the De-Giorgi's method to the solution $\theta : (0,\infty )\times \mathbb{R}^{N} \rightarrow \mathbb{R}$ of the modified critical Burger's equation \eqref{Artificalequation}. To begin, let $M > 1$ be an arbitrary large positive number (to be chosen later).  We consider the following sequence of truncations
\begin{equation*}
\theta_{k} = [ \theta - M (1-\frac{1}{2^{k}})]_{+},\quad k\geqslant 0.
\end{equation*}
By multiplying \eqref{Artificalequation} by $\theta_{k}$, and then taking integral over $\mathbb{R}^{N}$, we obtain
\begin{equation}\label{baseforDeGiorgi}
\frac{1}{2} \int_{\mathbb{R}^{N}} \partial_{t} [\theta_{k}^{2}] dx  + 
\epsilon \int_{\mathbb{R}^{N}} |\nabla \theta_{k}|^{2} dx = - \int_{\mathbb{R}^{N}} (-\triangle )^{\frac{1}{2}}\theta \cdot \theta_{k} dx , 
\end{equation}
in which we no longer see the term $ \sum_{j=1}^{N} \int_{\mathbb{R}^{N}}  \psi_{R}(\theta )(\partial_{j} \theta) \theta_{k}  dx$, thanks to the vanishing property \eqref{vanishing}.\\

To manage the term $  - \int_{\mathbb{R}^{N}} (-\triangle )^{\frac{1}{2}}\theta \cdot \theta_{k} dx $, it is necessary to use a recent result of C\'ordoba and C\'ordoba \cite{CorobodaDouble}, which states that for any convex function $\phi : \mathbb{R} \rightarrow \mathbb{R}$, we have
\begin{equation}\label{mCC}
-\phi '(\theta ) \cdot  (-\triangle )^{\frac{1}{2}}\theta \leqslant - (-\triangle )^{\frac{1}{2}}(\phi(\theta ) ) .
\end{equation}
To employ such a result, we consider the convex function
\[
\phi_{k}(\lambda ) = [\lambda - M(1-\frac{1}{2^{k}})]_{+}\quad\mbox{with}\quad \phi_{k}'(\lambda ) = \chi_{\{ \lambda > M( 1 - \frac{1}{2^{k}} )      \}}.
\] 
Then it follows from \eqref{mCC} 
\begin{equation*}
-[ (-\triangle )^{\frac{1}{2}}\theta  ]\cdot \theta_{k} = -\phi_{k}'(\theta )\cdot (-\triangle )^{\frac{1}{2}}\theta \cdot \theta_{k} \leqslant - [ (-\triangle )^{\frac{1}{2}}\theta_{k}  ]\cdot \theta_{k}.
\end{equation*}
We use this in \eqref{baseforDeGiorgi} to get
\begin{equation}\label{baseforDeGiorgitwo}
\frac{1}{2} \int_{\mathbb{R}^{N}} \partial_{t} [\theta_{k}^{2}] dx  + 
\epsilon \int_{\mathbb{R}^{N}} |\nabla \theta_{k}|^{2} dx \leqslant - \int_{\mathbb{R}^{N}} \theta_{k}\cdot (-\triangle )^{\frac{1}{2}}\theta_{k}  dx . 
\end{equation}
\indent Recall we wish to prove $\theta$ is $L^{\infty}$-bounded over $[t_{0}, \infty ] \times \mathbb{R}^{N}$ for every $t_{0} > 0$. Let $t_{0} > 0$ be fixed, and consider the increasing sequence 
\[
T_{k} = t_{0}(1-\frac{1}{2^{k}}), k\geqslant 0, 
\]
which approaches the limiting value $t_{0}$ as $k\rightarrow +\infty$.  Also fix $\sigma$ and $t$ verifying $T_{k-1} \leqslant \sigma \leqslant T_{k} \leqslant t < \infty$.  We then integrate \eqref{baseforDeGiorgitwo} over $ [\sigma , t]$ to obtain
\begin{equation}\label{baseDeGiorgithree}
\frac{1}{2} \int_{\mathbb{R}^{N}} \theta_{k}^{2}(t,\cdot )dx 
+ \int_{\sigma}^{t} \int_{\mathbb{R}^{N}} \theta_{k} \cdot (-\triangle)^{\frac{1}{2}}\theta_{k} dx ds
\leqslant \frac{1}{2}  \int_{\mathbb{R}^{N}} \theta_{k}^{2}(\sigma ,\cdot ) dx ,
\end{equation}
\emph{in which we purposely drop the artificial energy term $\epsilon \int_{\sigma}^{t}\int_{\mathbb{R}^{N}} |\nabla \theta_{k}|^{2} dx ds$, since we should not use it in estimating $\|\theta \|_{L^{\infty}([t_{0},\infty )\times \mathbb{R}^{N})}$}. Next, by taking the average over $\sigma \in [T_{k-1} , T_{k}]$ among the terms in the above inequality and then taking the $\sup$ over $t \in [T_{k} , \infty )$, we have 
\begin{equation*}
\frac{1}{2} \sup_{t \in [T_{k} , \infty )} \int_{\mathbb{R}^{N}} \theta_{k}^{2}(t,\cdot )dx 
+ \int_{T_{k}}^{\infty} \int_{\mathbb{R}^{N}} \theta_{k} \cdot (-\triangle)^{\frac{1}{2}}\theta_{k} dx ds
\leqslant \frac{2^{k}}{2t_{0}}  \int_{T_{k-1}}^{T_{k}} \int_{\mathbb{R}^{N}} \theta_{k}^{2}(\sigma ,\cdot ) dx d\sigma .
\end{equation*}
We now consider the following sequence of quantities
\begin{equation*}
U_{k} =  \sup_{t \in [T_{k} , \infty )} \int_{\mathbb{R}^{N}} \theta_{k}^{2}(t,\cdot )dx 
+ 2 \int_{T_{k}}^{\infty} \int_{\mathbb{R}^{N}} \theta_{k} \cdot (-\triangle)^{\frac{1}{2}}\theta_{k} dx ds.
\end{equation*}
Then, our last inequality tells us that
\begin{equation}\label{baseDeGirogifour}
U_{k} \leqslant  \frac{2^{k}}{t_{0}}  \int_{T_{k-1}}^{\infty} \int_{\mathbb{R}^{N}} \theta_{k}^{2}(\sigma ,\cdot ) dx d\sigma .
\end{equation}
Our goal is to build up a nonlinear recurrence relation for $U_{k}, k\geqslant 0$ by relying on the above inequality. By employing Sobolev embedding, interpolation and H\"older's inequality we know that our solution $\theta$ of \eqref{Artificalequation} satisfies the following inequality for all $k \geqslant 1$
\begin{equation*}
\|\theta_{k-1}\|_{L^{2(1+\frac{1}{N})}(  [T_{k-1} ,\infty )\times \mathbb{R}^{N} )} \leqslant C U_{k-1}^{\frac{1}{2}} ,  
\end{equation*}
for some contant $C$ depending only on $N$. Because of this, we can raise up the index for $\int_{T_{k-1}}^{\infty} \int_{\mathbb{R}^{N}} \theta_{k}^{2}(\sigma ,\cdot ) dx d\sigma  $ as follows
\begin{equation*}
\int_{T_{k-1}}^{\infty} \int_{\mathbb{R}^{N}} \theta_{k}^{2}dx d\sigma
\leqslant \int_{T_{k-1}}^{\infty} \int_{\mathbb{R}^{N}} \theta_{k}^{2} \chi_{\{ \theta_{k-1} > \frac{M}{2^{k}}         \}}
\leqslant \left(\frac{2^{k}}{M}\right)^{\frac{2}{N}} \int_{T_{k-1}}^{\infty} \int_{\mathbb{R}^{N}} \theta_{k}^{2} \theta_{k-1}^{\frac{2}{N}} 
\leqslant \left(\frac{2^{k}}{M}\right)^{\frac{2}{N}} C U_{k-1}^{1+\frac{1}{N}} .
\end{equation*}
Hence \eqref{baseDeGirogifour} together with our last inequality gives
\begin{equation}\label{baseDeGiorgifive}
U_{k} \leqslant \frac{2^{k(1+\frac{2}{N})}}{t_{0} M^{\frac{2}{N}}} C U_{k-1}^{1+\frac{1}{N}} .
\end{equation}
We can now choose $M = (\frac{1}{t_{0}})^{\frac{N}{2}}$ so that $t_{0} M^{\frac{2}{N}} = 1$. Hence
\be\label{baseDeGiorgisix}
U_{k} \leqslant 2^{k(1+\frac{2}{N})} C U_{k-1}^{1+\frac{1}{N}},\quad k \geqslant 1.
\ee
From the nonlinear recurrence relation \eqref{baseDeGiorgisix}, we know that there exists some constant $\delta_{N} \in (0,1)$, depending only on $N$, such that $U_{k} \rightarrow 0$ as $k\rightarrow \infty$, provided we have $U_{1} < \delta_{N}$. Due to this observation, if the initial datum $\theta_{0} = \theta (0,\cdot )$ verifies
\[
\|\theta_{0}\|_{L^{2}(\mathbb{R}^{N})} < 
(\frac{\delta_{N}}{2^{1+\frac{2}{N}}C}   )^{\frac{N}{N+1}},
\]
then, we must have that 
\begin{equation*}
U_{1} \leqslant 2^{1+\frac{2}{N}}C U_{0}^{1+\frac{1}{N}} 
\leqslant 2^{1+\frac{2}{N}}C  \|\theta_{0}\|_{L^{2}(\mathbb{R}^{N})}^{1+\frac{1}{N}} < \delta_{N}.
\end{equation*}
Note we use $U_0\leq \|\theta_0\|^2_{L^2(R^N)}$, which holds because of the energy inequality that can be obtained in a standard way for the Leray-Hopf solutions of \eqref{Artificalequation}. 
For such a $\theta_0$, we have $\lim_{k\rightarrow \infty} U_{k} = 0$, and hence $\theta \leqslant M = (\frac{1}{t_{0}})^{\frac{N}{2}}$ is valid almost everywhere on $[t_{0} , \infty )\times \mathbb{R}^{N}$. By applying the same De-Giorgi's method to $-\theta$, we should also get $-\theta \leqslant M = (\frac{1}{t_{0}})^{\frac{N}{2}}$ almost everywhere on $[t_{0} , \infty )\times \mathbb{R}^{N}$. At this point, let us summarize what we have done so far:
\begin{itemize}
\item If $\theta : (0,\infty )\times \mathbb{R}^{N} \rightarrow \mathbb{R}$ is a weak solution of the modified critical Burgers' equation with initial datum $\theta (0,\cdot ) = \theta_{0} \in L^{2}( \mathbb{R}^{N}  )$  verifying $\|\theta_{0}\|_{L^{2}(\mathbb{R}^{N})} < 
(\frac{\delta_{N}}{2^{1+\frac{2}{N}}C}   )^{\frac{N}{N+1}}$, then it follows that 
$\|\theta \|_{L^{\infty}([t_{0} , \infty )\times \mathbb{R}^{N}  )} \leqslant (\frac{1}{t_{0}})^{\frac{N}{2}}$ for every $t_{0} > 0$.
\end{itemize}
Next, we need to remove the smallness condition imposed on $\|\theta_{0}\|_{L^{2}(\mathbb{R}^{N})}$ in the above statement. To this end, let $\theta : (0,\infty )\times \mathbb{R}^{N} \rightarrow \mathbb{R}$ be a given weak solution of the modified critical Burgers' equation \eqref{Artificalequation}, and let $\lambda > 0$ be the unique positive number such that
\begin{equation}\label{lambdadef}
\frac{1}{\lambda^{\frac{N}{2}}} \|\theta_{0}\|_{L^{2}(\mathbb{R}^{N})}
= \frac{1}{2}  \left(\frac{\delta_{N}}{2^{1+\frac{2}{N}}C}   \right)^{\frac{N}{N+1}}. 
\end{equation}
For such a $\lambda > 0$, we consider the rescaled function $\theta_{\lambda}(t,x) = \theta (\lambda t, \lambda x)$, which solves the following rescaled modified Burgers' equation in the weak sense
\begin{equation*}
 \partial_{t}\theta_{\lambda} + \sum_{j=1}^{N} \psi_{R}(\theta_{\lambda} ) \cdot \partial_{j} \theta_{\lambda} = -(-\triangle)^{\frac{1}{2}} \theta_{\lambda}  + \frac{\epsilon}{\lambda} \triangle \theta_{\lambda}. 
\end{equation*}
At first glance, it seems to be troublesome that $\theta_{\lambda}$ no longer solves the original equation \eqref{Artificalequation}.  However, this is not problematic at all since the energy term 
$\epsilon \int_{\mathbb{R}^{N}} |\nabla \theta_{k}|^{2} dx$ is purposely dropped from inequality \eqref{baseforDeGiorgitwo} before we apply the De-Giorgi's method to $\theta_{k}$.  This means that all the estimates starting from \eqref{baseDeGiorgithree} in the above process are independent of the artificial diffusion term $\epsilon \triangle \theta$. This tells us, in particular that if $\theta_{k}$ is replaced by $[\theta_{\lambda} - M(1 - \frac{1}{2^{k}} ) ]_{+}$ in inequality \eqref{baseDeGiorgithree}, all the estimates thereafter remain unchanged.  This observation, together with the fact that by \eqref{lambdadef}
\[ 
\|\theta_{\lambda} (0,\cdot ) \|_{L^{2}(\mathbb{R}^{N})} < \frac{1}{2} \left(\frac{\delta_{N}}{2^{1+\frac{2}{N}}C}   \right)^{\frac{N}{N+1}},  \]
give us
\begin{equation*}
\|\theta_{\lambda} \|_{L^{\infty}([\frac{t_{0}}{\lambda} , \infty ) \times \mathbb{R}^{N})}
\leqslant \frac{1}{(\frac{t_{0}}{\lambda })^{\frac{N}{2}}}, \quad  t_{0} > 0.
\end{equation*}
Since $ \|\theta_{\lambda} \|_{L^{\infty}([\frac{t_{0}}{\lambda} , \infty ) \times \mathbb{R}^{N})} =
\|\theta \|_{L^{\infty}([t_{0} , \infty ) \times \mathbb{R}^{N})}      $, it follows from \eqref{lambdadef} that the following inequality is valid for every $t_{0} > 0$
\begin{equation*}
\|\theta \|_{L^{\infty}([t_{0} , \infty ) \times \mathbb{R}^{N})}
\leqslant 2 \left( \frac{2^{1+\frac{2}{N}} C }{\delta_{N}}    \right)^{\frac{N}{N+1}} 
\frac{ \|\theta_{0}\|_{L^{2}(\mathbb{R}^{N})}  }{t_{0}^{\frac{N}{2}} } .
\end{equation*}
In summary, we have established 
\begin{itemize}
\item There exists some universal constant $C_{N} \in (0,\infty )$, depending only on $N$, such that for every weak solution $\theta^{(\epsilon , R)} : (0,\infty )\times \mathbb{R}^{N} \rightarrow \mathbb{R}$ of the modified critical Burgers' equation \eqref{Artificalequation} with initial datum $\theta^{(\epsilon ,R)} (0,\cdot) = \theta_{0} \in L^{2}(\mathbb{R}^{N})$, we have 
$\|\theta^{(\epsilon , R)} \|_{L^{\infty}([t_{0} , \infty ) \times \mathbb{R}^{N})}
\leqslant C_{N}\cdot \frac{ \|\theta_{0}\|_{L^{2}(\mathbb{R}^{N})}  }{t_{0}^{\frac{N}{2}} }  $, for every $t_{0} > 0$.
\end{itemize}
Now, the solution $\theta^{(\epsilon , R)}$ of the modified critical Burgers' equation \eqref{Artificalequation} satisfies the uniform bound $C_{N}\frac{ \|\theta_{0}\|_{L^{2}(\mathbb{R}^{N})}  }{t_{0}^{\frac{N}{2}} }$. By passing to the limit, as $\epsilon \rightarrow 0^{+}$ and $R \rightarrow +\infty$, it follows that $ \theta^{(\epsilon , R)}$ converges to some weak solution $\theta : (0,\infty )\times \mathbb{R}^{N} \rightarrow \mathbb{R}$ of the critical Burgers' equation \eqref{Burger}, which must also satisfy the same uniform bound $C_{N}\frac{ \|\theta_{0}\|_{L^{2}(\mathbb{R}^{N})}  }{t_{0}^{\frac{N}{2}} }$. So, we are finished with the proof of Theorem~\ref{weaksolutiontheorem}.
\end{proof}


\section{Harmonic extension to $\mathbb{R}^{N}\times [0,\infty )$ and the local energy inequality}\label{leisection}
We begin by introducing the harmonic extension (See \cite{Driftdiffusion} and \cite{CaffarelliSilvestre} for more details).  Operator $(-\triangle)^\frac{1}{2} \theta$ is not a local operator.  However it can be localized.  Indeed, define the harmonic extension operator $H: C^\infty_0(\R^N)\mapsto C^\infty_0(\R^N\times \R^+)$ by
\begin{align*}
&-\triangle H(\theta)=0\quad\mbox{in}\quad \R^N \times(0,\infty),\\
&H(\theta)(x,0)=\theta(x),\quad x\in \R^N.
\end{align*}
Then it can be shown we can view $(-\triangle)^\frac{1}{2} \theta$ as the normal derivative of $H(\theta)$ on the boundary $\{(x,0): x \in \R^N\}$ i.e.,
\[
(-\triangle)^\frac{1}{2}\theta(x)=-\partial_\nu H(\theta)(x).
\]
From now on we use $\theta^*$ to denote the harmonic extension of $\theta$ or more precisely
\[
\theta^*(t,x,z)=H(\theta(t,\cdot))(x,z).
\]
Now we are ready to proceed to the local energy inequality and its proof, which closely follows \cite{Driftdiffusion}.
\begin{prop}\label{LEI}
(Local Energy Inequality) Let $\theta : [-4,0]\times \mathbb{R}^{N}\rightarrow \mathbb{R}$ be a weak solution of the Burgers' equation \eqref{Burger}. Then, for any function $u$ in the form of $u = \beta [\theta -L]$, with $\beta > 0$, and $L \in \mathbb{R}$ we have 
\begin{equation*}
\begin{split}
&\frac{1}{2}\int_{\sigma}^{t} \int_{B_{4}^{*}} |\nabla (\eta u_{+}^{*})|^{2} dx dz ds
+ \frac{1}{2} \int_{B_{4}} (\eta u_{+})^{2}(t,x) dx\\
&\quad \leqslant \frac{1}{2} \int_{B_{4}} (\eta u_{+})^{2}(\sigma ,x) dx 
 + 2N C_{N} \left(|L| + \|\theta\|_{L^{\infty}([-4,0]\times \mathbb{R}^{3})} \right)^{2}
\int_{\sigma}^{t}\int_{B_{4}}|\nabla \eta |^{2} u_{+}^{2}dxds\\
&\quad\quad+ \int_{\sigma}^{t}\int_{B_{4}^{*}} |\nabla \eta |^{2} (u_{+}^{*})^{2} dx dz dt,
\end{split}
\end{equation*}
where $\eta$ can be any cut off function supported in $B_{4}^{*}=B_4\times [0,4], B_4=[-4,4]^N$, and $C_{N}$ is the constant appearing in the Sobolev inequality $\|f\|_{L^{\frac{2N}{N-1}}(\mathbb{R}^{N})}^{2} \leqslant C_{N} \|f\|_{\dot H^{\frac{1}{2}}(\mathbb{R}^{N})}^{2}$.
\end{prop}
\begin{proof}
Start with
\begin{align}
0&=\int_{B^*_4}\eta^2u^*_+\triangle u^\ast dx dz\nonumber\\
&=-\int_{B^*_4}\nabla(\eta^2u^*_+)\cdot\nabla u^\ast dx dz+\int_{B_4}\eta^2u^*_+\partial_zu^*dx\Big\rvert^4_0\nonumber\\
&=-\int_{B^*_4}\nabla(\eta^2u^*_+)\cdot\nabla u^\ast dx dz+\int_{B_4}\eta^2(x,0)u_+(-\triangle)^\frac{1}{2} u dx\label{m3},
\end{align}
where we use
\[
(-\triangle)^\frac{1}{2} u =-\partial_z u^*(\cdot,0).
\]
A calculation shows that \eqref{m3} is equivalent to 
\be
0=-\int_{B^*_4}|\nabla(\eta u^*_+)|^2dxdz+\int_{B^*_4}|\nabla \eta|^2[u^*]_+^2 dx dz+\int_{B_4}\eta^2(x,0)u_+(-\triangle)^\frac{1}{2} u dx\label{m4}.
\ee
Now if $\theta$ solves \eqref{Burger}, $u$ solves
\be
\partial_t u + \sum^N_{j=1}\frac{1}{\beta}(u+L\beta)\partial_j u=-(-\triangle)^\frac{1}{2} u.\label{m1}
\ee
Also observe
\[
3\int_{B_4}\eta^2u_+\frac{1}{\beta}(u+L\beta)\partial_j udx=-\int_{B_4}\partial_j(\eta^2)[u_+]^2\frac{1}{\beta}(u+L\beta)dx
+L\int_{B_4} \eta^2 u_+\partial_judx.
\]
Hence for the third term on the RHS in \eqref{m4} we have
\be\label{m2}
\begin{split}
&-\int_{B_4} \eta^2u_+(-\triangle)^\frac{1}{2} u
=\partial_t\left(\int_{B_4}\eta^2\frac {u_+^2}{2}dx\right)-\frac{1}{3} \sum^N_{j=1}\int_{B_4}\partial_j(\eta^2)u_+^2\frac{1}{\beta}(u+L\beta)dx\\
&\quad\quad+\frac{L}{3}\sum^N_{j=1}\int_{B_4}\eta^2 u_+\partial_judx.
\end{split}
\ee
Substitute \eqref{m2} into \eqref{m4}, integrate between $\sigma$ and $t$, and take the absolute value of the RHS to obtain
\be\label{m5}
\begin{split}
&\int^t_\sigma\int_{B^*_4}|\nabla(\eta u^*_+)|^2dxdzds
+\int_{B_4}\eta^2\frac {u_+^2(t)}{2}dx\\
&\quad\leq\int^t_\sigma\int_{B^*_4}|\nabla \eta|^2[u^*]_+^2 dx dzds+\int_{B_4}\eta^2\frac {u_+^2(\sigma)}{2}dx\\
&\quad\quad+\frac{1}{3} \Big\lvert\sum^N_{j=1}\int_{B_4}\partial_j(\eta^2)u_+^2\frac{1}{\beta}(u+L\beta)dx\Big\rvert
+\Big\lvert\frac{L}{3}\sum^N_{j=1}\int_{B_4}\eta^2 u_+\partial_judx\Big\rvert.
\end{split}
\ee
We examine the last two terms.  Both can be written as a constant multiple of 
\[
\Big\lvert\sum^N_{j=1}\int^t_\sigma\int_{B_4}\eta\partial_j\eta u_+^2vdxds\Big\rvert,
\]
where $v=\frac{1}{\beta}(u+L\beta)=\theta$ for the first term and $v=L$ for the second.  Following \cite{Driftdiffusion} by H\"older's inequality in space and Cauchy's inequality with $\epsilon$ in time we obtain
\begin{align*}
\Big\lvert\sum^N_{j=1}\int^t_\sigma\int_{B_4}\eta\partial_j\eta u_+^2vdxds\Big\rvert
\leq \frac{1}{\epsilon}\int^t_\sigma \|\eta u_+\|^2_{L^\frac{2N}{N-1}}ds+\epsilon\int^t_\sigma \|\nabla\eta v u_+\|^2_{L^\frac{2N}{N+1}}ds.
\end{align*}
By the arguments on top of p.8 in \cite{Driftdiffusion}
\[
\frac{1}{\epsilon}\int^t_\sigma\|\eta u_+\|^2_{L^\frac{2N}{N-1}(\R^N)}ds\leq \frac{1}{\epsilon}C_N \int^t_\sigma\int_{B^*_4}|\nabla (\eta u^*_+)|^2 dx dzds,
\]
which means it can be combined with the LHS of \eqref{m5} if $\epsilon$ is small enough.  Next,
since $\frac{2N}{N+1}\leq 2$ and $\eta$ has compact support within $B^*_4$ we have
\begin{align*}
\epsilon\int^t_\sigma \|\nabla\eta vu_+\|^2_{L^\frac{2N}{N+1}}ds&\leq C\epsilon \int^t_\sigma\int_{B_4} |\nabla\eta vu_+|^2dxds\\
&\leq C\epsilon \|v\|^2_{L^\infty}\int^t_\sigma\int_{B_4} |\nabla\eta u_+|^2dxds\\
&\leq C\epsilon (|L|+\|\theta\|_{L^\infty})^2\int^t_\sigma\int_{B_4} |\nabla\eta u_+|^2dxds
\end{align*}
as needed.
\end{proof}

\section{Proof of Theorem 1}\label{mainproof}
Theorem 2 proven in section \ref{weaksolutions} gives us the first part of Theorem 1.  What remains is to establish the H\"older continuity for solutions of equation \eqref{Burger}.  For this purpose, we need the following three lemmas.  In what follows, we use the abbreviations that $Q_{r}^{*} = [-r,0]\times B_{r}^{*}$ and $Q_{r} = [-r,0]\times B_{r}$.
\begin{lem}\label{Lemma1}
Given any $C_{\theta} \in (0, \infty )$, there exists some $\epsilon_{0} > 0$ (depending only on $N$ and $C_{\theta}$), and some $\lambda \in (0, \frac{1}{2})$(depending only on $N$), such that for every solution $\theta : [-\frac{1}{4} ,0]\times \mathbb{R}^{N}\rightarrow \mathbb{R}$ of equation \eqref{Burger} which verifies $\|\theta \|_{L^{\infty}([-\frac{1}{4},0]\times \mathbb{R}^{N})} \leqslant C_{\theta}$, we have the following implication for every function $u$ in the form of $u = \beta [\theta -L]$, with $|\beta | \geqslant min\{1, \frac{1}{C_{\theta }}\}$, and $|L| \leqslant 6 C_{\theta}$\\
\begin{itemize}
\item If $u^{*} = \beta [\theta^{*} - L]$ verifies $u^{*}\leqslant 2$ on $[-\frac{1}{4},0]\times B_{\frac{1}{4}}^{*}$, and 
\[
\int_{-\frac{1}{4}}^{0}\int_{B_{\frac{1}{4}}^{*}} (u_{+}^{*})^{2} + \int_{-\frac{1}{4}}^{0}\int_{B_{\frac{1}{4}}} u_{+}^{2} \leqslant \epsilon_{0}, 
\]
then it follows that  $u \leqslant 2-\lambda$ on $[-\frac{1}{16},0]\times B_{\frac{1}{16}}$.
\end{itemize}
\end{lem}
\begin{lem}\label{Lemma2}
Given any $C_{\theta} \in (0,\infty )$, and any sufficiently small $\epsilon_{1} > 0$, there exists some $\delta_{1} > 0$, and also some constant $D_{\theta} \in (0,\infty )$ ( depending only on $C_{\theta}$ and $N$), such that for every solution $\theta : [-1,0]\times \mathbb{R}^{N} \rightarrow \mathbb{R}$ of equation \eqref{Burger} which verifies $\|\theta \|_{L^{\infty}([-1,0]\times\mathbb{R}^{N})} \leqslant C_{\theta}$, we have the following implication for all function $u$ in the form of $u=\beta [\theta - L]$, with $|\beta | \geqslant  \frac{1}{C_{\theta}} $, and $|L| \leqslant 6 C_{\theta}$\\
\begin{itemize}
\item If $u^{*} = \beta [\theta^{*} - L]$ verifies the following three conditions:
\begin{itemize}
\item [i)] $u^{*} \leqslant 2$ on $Q_{1}^{*} = [-1,0]\times B_{1}^{*}$,
\item [ii)]$|\{  (t,x,z)\in Q_{1}^{*} : u^{*}(t,x,z) \leqslant 0  \} | \geqslant \frac{|Q_{1}^{*}|}{2} $,
\item [iii)]$| \{(t,x,z) \in Q_{1}^{*} :  0 < u^{*}(t,x,z) < 1 \} | \leqslant \delta_{1}$ , 
\end{itemize}
then it follows that 
\be\label{mlemm1}
\int_{Q_{\frac{1}{4}}} [u-1]_{+}^{2} dx dt  + \int_{Q_{\frac{1}{4}}^{*}} [u^{*}-1]_{+}^{2} dx dz dt \leqslant D_{\theta } \epsilon_{1}^{\frac{1}{2}}.
\ee
\end{itemize}
\end{lem}
\begin{lem}\label{Oscillation}
(Oscillation Lemma) Given any $C_{\theta } \in (0,\infty )$, there exists some $\lambda^{*} > 0 $ (depending only on $N$ and $C_{\theta }$), such that for every solution $\theta : [-1,0] \times \mathbb{R}^{N} \rightarrow \mathbb{R}$ of equation \eqref{Burger}, which verifies $\|\theta \|_{L^{\infty}([-1,0]\times \mathbb{R}^{N})} \leqslant C_{\theta }$, we have the following implication for any function $u$ in the form of $u = \beta [\theta - L] $, with $|\beta | \geqslant \frac{1}{C_{\theta } } $, and $|L| \leqslant 3 C_{\theta }$\\
\begin{itemize}
\item If it happens that $u^{*} \leqslant 2$ on $Q_{1}^{*}$, and 
$| \{ (t,x,z) \in Q_{1}^{*} : u^{*} \leqslant 0  \}  | \geqslant \frac{|Q_1^*|}{2}$, then it follows that 
$u^{*} \leqslant 2-\lambda^{*}$ on $Q_{\frac{1}{32}}^{*}$. 
\end{itemize}
\end{lem}
\begin{rem}  The above lemmas correspond to Lemma 6, Lemma 8 and Proposition 9 in \cite{Driftdiffusion} respectively.  However, here they are not stated for the solution $\theta$ of the equation, but for the function $u=\beta[\theta-L]$ since this is the function that we actually apply the lemmas to.  Most of all, the above lemmas require restrictions for the constants $\beta$ and $L$, which were not needed in \cite{Driftdiffusion}.  This is a result of the main difficulties of dealing with the Burgers' equation explained in the introduction.
\end{rem}
\begin{rem}
The proof of Lemma~\ref{Lemma1} and Lemma~\ref{Lemma2} relies on the local energy inequality as established in Proposition~\ref{LEI}.  The two lemmas are technical tools needed to establish the Oscillation Lemma and are proven at the end of this section. 
\end{rem}
\begin{rem}
The Oscillation Lemma gives us the H\"older continuity.
We describe this next, and then give the proof of the Oscillation Lemma.  It is very important to observe that the universal constant $\lambda^{*}$ in the Oscillation Lemma is invariant under the natural scaling $\theta_{\lambda}(t,x) = \theta (\lambda t , \lambda x)$ for solutions of the N-dimensional critical Burgers' equation.  This observation is of crucial importance since it allows us to employ the Oscillation Lemma at different scales in the proof of H\"older continuity (see below).  The scale-invariant property of the Oscillation Lemma is due to the invariance of solutions for the $N$-dimensional critical Burgers' equation under above scaling.  This in particular explains why our method works in the critical case.
\end{rem}
\noindent
\textbf{Proof of H\"older Continuity:}  Set $r=\frac{1}{32}$.  Let $\theta : [-1,0]\times \mathbb{R}^{N} \rightarrow \mathbb{R}$ be a solution of \eqref{Burger} with  $\|\theta \|_{L^{\infty}([-1,0]\times \mathbb{R}^{N})} \leqslant C_{\theta}$ for some $C_{\theta } \in (0,\infty ).$  In order to use Lemma~\ref{Oscillation}, it is necessary to consider the function $u_{0} = \beta_{0} \{\theta - L_{0}\}$, where the constants $\beta_{0}$ and $L_{0}$ are given by
\[
\beta_{0} = \frac{2}{\sup_{Q_{1}^{*}}\theta^{*} - \inf_{Q_{1}^{*}} \theta^{*}}\quad\mbox{and}\quad 
L_{0} = \frac{\sup_{Q_{1}^{*}}\theta^{*} + \inf_{Q_{1}^{*}} \theta^{*}  }{2}.
\]
Note $\beta_{0} \geqslant \frac{1}{C_{\theta }}$, and $|L_{0}|\leqslant C_{\theta}$. Now, we are going to construct a sequence of functions $u_{k} = \beta_{k} \{\theta - L_{k}\}$ inductively in a way that is dependent on $u_{0}$.\\
\indent To begin the inductive process, we observe that $u_{0} = \beta_{0} \{\theta - L_{0}\}$ verifies the condition that $|u_{0}^{*}| \leqslant 2$ on $Q_{1}^{*}$. To construct a suitable $u_{1} = \beta_{1} \{ \theta -L_{1} \}$ from $u_{0}$, we split our discussion into two cases:\newline\noindent
\textbf{{Case 1:  $| \{ (t,x,z) \in Q_{1}^{*} : u_{0}^{*} \leqslant 0  \}  | \geqslant \frac{|Q_{1}^{*}|}{2}$}}.  We apply Lemma~\ref{Oscillation} to $u_{0}^{*}$ over $Q_{1}^{*}$ and deduce that $-2 \leqslant u_{0}^{*} \leqslant 2-\lambda^{*}$ on $Q_{r}^{*}$, where $r=\frac{1}{32}$ and $\lambda^{*}$ is the constant in Lemma~\ref{Oscillation}.  Hence, we have
\[
\Big\lvert\frac{2}{2-\frac{\lambda^{*}}{2}} \{u_{0}^{*} + \frac{\lambda^{*}}{2}  \}\Big\rvert \leqslant 2\quad\mbox{on}\quad Q_{r}^{*}.
\]
Let $a =\frac{2}{2-\frac{\lambda^{*}}{2}} $, and define $u_{1}$ to be
\[
u_{1} = a \{u_{0} + \frac{\lambda^{*}}{2}   \} 
= a\beta_{0} \{ \theta - L_{0} + \frac{\lambda^{*}}{2\beta_{0}}   \}.
\]
\newline\noindent
\textbf{{Case 2: $| \{ (t,x,z) \in Q_{1}^{*} : -u_{0}^{*} \leqslant 0  \}  | \geqslant \frac{|Q_{1}^{*}|}{2}$}}.  In this case, we apply Lemma~\ref{Oscillation} to $-u_{0}^{*}$ over $Q_{1}^{*}$ and deduce that $-2 \leqslant -u_{0}^{*} \leqslant 2-\lambda^{*}$ on $Q_{r}^{*}$. Hence, we have
\[
\Big\lvert\frac{2}{2-\frac{\lambda^{*}}{2}} \{u_{0}^{*} - \frac{\lambda^{*}}{2}  \}\Big\rvert \leqslant 2 \quad\mbox{on}\quad Q_{r}^{*}.
\]
As before, we write $a =\frac{2}{2-\frac{\lambda^{*}}{2}} $, and define in this case that 
\[
u_{1} = a \{u_{0} - \frac{\lambda^{*}}{2}   \}= a\beta_{0} \{ \theta - L_{0} - \frac{\lambda^{*}}{2\beta_{0}}   \}.
\]
\newline
We observe that in either case
\begin{itemize} 
\item $u_{1} = a\beta_{0} \{ \theta - L_{0} + (-1)^{\sigma_{1}}\frac{\lambda^{*}}{2\beta_{0}}   \},\quad \sigma_{1} \in \{0,1\}.$
\item $|u_{1}^{*}| \leqslant 2$ on $Q_{r}^{*}$ .
\item $ |a\beta_{0}| \geqslant \frac{1}{C_{\theta}}$, and 
$ |L_{0} -(-1)^{\sigma_{1}}\frac{\lambda^{*}}{2\beta_{0}}| \leqslant C_{\theta} + \frac{\lambda^{*}}{2\beta_{0}} \leqslant \frac{3}{2} C_{\theta}           $ .
\end{itemize}
This means that we can apply Lemma~\ref{Oscillation} to $u_{1}^{*}$ over $Q_{r}^{*}$ in order to construct $ u_{2} = a \{u_{1} + (-1)^{\sigma_{2}} \frac{\lambda^{*}}{2}   \}$ in exactly the same way.\\
\indent For the reasons of transparency and completeness we descibe now the inductive step.  Suppose that at step $k\in \mathbb{N}^{+}$, we have a function $u_{k}$ given by
\begin{equation*}
u_{k} = a^{k} \beta_{0} \{ \theta - L_{0} + \frac{\lambda^{*}}{2\beta_{0}} \sum_{s=1}^{k} (-1)^{\sigma_{s}} 
(\frac{1}{a})^{s-1}                          \} ,
\end{equation*}
which verifies the required condition that
\[
|u_{k}^{*}| \leqslant 2\quad\mbox{on}\quad Q_{r^{k}}^{*}.
\]
Here, let us make the crucial observation that 
\begin{itemize}
\item $a^{k} \beta_{0} \geqslant \beta_{0} \geqslant \frac{1}{C_{\theta }}$.
\item $  | L_{0} - \frac{\lambda^{*}}{2\beta_{0}} \sum_{s=1}^{k} (-1)^{\sigma_{s}} 
(\frac{1}{a})^{s-1} | \leqslant C_{\theta} +   \frac{\lambda^{*}}{2\beta_{0}} \frac{1}{(1-\frac{1}{a})}
= C_{\theta} +  \frac{\lambda^{*}}{2\beta_{0}} \frac{4}{\lambda^{*}} \leqslant 3 C_{\theta}           $, 
\end{itemize}
where in the second term, we have implicitly used the fact that $\frac{1}{(1-\frac{1}{a})} = \frac{4}{\lambda^{*}}$.  The above two inequalities simply tell us that we can apply Lemma~\ref{Oscillation} to $u_{k}^{*}$ over $Q_{r^{k}}^{*}$ in either one of the following two cases:\\
\textbf{{Case 1:  $| \{ (t,x,z) \in Q_{r^k}^{*} : u_{k}^{*} \leqslant 0  \}  | \geqslant \frac{|Q_{r^k}^{*}|}{2}$}}.  We apply Lemma~\ref{Oscillation} to $u_{k}^{*}$ over $Q_{r^{k}}^{*}$ and deduce that $-2 \leqslant u_{k}^{*} \leqslant 2-\lambda^{*}$ on $Q_{r^{k+1}}^{*}$. Hence we have $|\frac{2}{2-\frac{\lambda^{*}}{2}} \{u_{k}^{*} + \frac{\lambda^{*}}{2}  \}| \leqslant 2 $ on $Q_{r^{k+1}}^{*}$.  We define $u_{k+1}$ as $u_{k+1} = a\{u_{k} + \frac{\lambda^{*}}{2}\}$. So, we have 
\begin{equation*}
u_{k+1} = a^{k+1} \beta_{0} \{ \theta - L_{0} + \frac{\lambda^{*}}{2\beta_{0}} \sum_{s=1}^{k} (-1)^{\sigma_{s}} (\frac{1}{a})^{s-1} +  \frac{\lambda^{*}}{2\beta_{0}} (\frac{1}{a})^{k} \}.
\end{equation*}\newline\noindent
\textbf{{Case 2: $| \{ (t,x,z) \in Q_{r^k}^{*} : -u_{k}^{*} \leqslant 0  \}  | \geqslant \frac{|Q_{r^k}^{*}|}{2}$}}.  We can apply Lemma~\ref{Oscillation} to  $-u_{k}^{*}$ over $Q_{r^{k}}^{*}$, and deduce that $-2 \leqslant -u_{k}^{*} \leqslant 2-\lambda^{*}$ over $Q_{r^{k+1}}^{*}$. Hence, we have $|\frac{2}{2-\frac{\lambda^{*}}{2}} \{u_{k}^{*} - \frac{\lambda^{*}}{2}  \}| \leqslant 2 $ on $Q_{r^{k+1}}^{*}$. Because of this, we define $u_{k+1} = a \{  u_{k}^{*} - \frac{\lambda^{*}}{2}  \}$. So, we have
\begin{equation*}
u_{k+1} = a^{k+1} \beta_{0} \{ \theta - L_{0} + \frac{\lambda^{*}}{2\beta_{0}} \sum_{s=1}^{k} (-1)^{\sigma_{s}} (\frac{1}{a})^{s-1} -  \frac{\lambda^{*}}{2\beta_{0}} (\frac{1}{a})^{k} \}.
\end{equation*}
\newline\indent
From the above inductive process, we have a sequence of functions 
\[
u_{k} = a^{k} \beta_{0} \{ \theta - L_{0} + \frac{\lambda^{*}}{2\beta_{0}} \sum_{s=1}^{k} (-1)^{\sigma_{s}} 
(\frac{1}{a})^{s-1}       \},
\]
which verify the following conditions
\begin{itemize}
\item  $|u_{k}^{*}| \leqslant 2$ on $Q_{r^{k}}^{*},$ for any  $k\geqslant 1$.
\item  $a^{k}\beta_{0} \geqslant \frac{1}{C_{\theta}}$, for any $k \geqslant 1$.
\item  $ | L_{0} - \frac{\lambda^{*}}{2\beta_{0}} \sum_{s=1}^{k} (-1)^{\sigma_{s}} 
(\frac{1}{a})^{s-1} | \leqslant 3 C_{\theta}    $, for any $k\geqslant 1$.
\end{itemize}
Therefore we can deduce for all $k \geqslant 1$
\begin{equation*}
 a^{k}\beta_{0} (  \sup_{Q_{r^{k}}^{*}}\theta^{*} - \inf_{Q_{r^{k}}^{*}} \theta^{*}   )
= \sup_{Q_{r^{k}}^{*}} u_{k}^{*} - \inf_{Q_{r^{k}}^{*}} u_{k}^{*} \leqslant 4. 
\end{equation*}
Thus
\begin{equation*}
 \sup_{Q_{r^{k}}^{*}}\theta^{*} - \inf_{Q_{r^{k}}^{*}} \theta^{*} \leqslant \frac{4}{\beta_{0}} (\frac{1}{a})^{k} \leqslant 4 C_{\theta } (\frac{1}{a})^{k}.
\end{equation*}
At this point, we note that the above inequality and the shift-invariant property of solutions of \eqref{Burger} give us the conclusion that $\theta^{*}$ is $C^{\alpha}$ at any $(t,x,z)$, and hence $\theta$ itself must be $C^{\alpha}$.  This completes the proof of Theorem 1.
\subsection{Proof of the Oscillation Lemma}
The proof closely follows \cite{Driftdiffusion}.  Assume Lemmas ~\ref{Lemma1} and ~\ref{Lemma2} hold, and let $\theta:[-1,0]\times \mathbb{R}^{N} \rightarrow \mathbb{R}$ be a solution to equation \eqref{Burger} with $\|\theta \|_{L^{\infty}([-1,0]\times \mathbb{R}^{N})} \leqslant C_{\theta }$ for some $C_{\theta} \in (0,\infty )$, as well as let $\epsilon_{0}$ (depending only on $N$ and $C_{\theta}$), and $\lambda \in (0,2)$ (depending only on $N$) be the two constants appearing in Lemma~\ref{Lemma1}. Also, consider the constant $D_{\theta}$ (depending only on $C_{\theta}$), which appears in Lemma~\ref{Lemma2}.  We choose $\epsilon_{1} = \{ \frac{1}{4D_{\theta}} \frac{\epsilon_{0}}{2}   \}^{2}$, so that we have $4D_{\theta} (\epsilon_{1})^{\frac{1}{2}} = \frac{\epsilon_{0}}{2} < \epsilon_{0}$. With such an $\epsilon_{1}$, we have a small number $\delta_{1}$ (depending only on $\epsilon_{1}$) as it appears in the statement of Lemma~\ref{Lemma2}.\\
\indent
With these preparations, let $u =\beta \{ \theta -  L  \}$, with $|\beta | \geqslant  \frac{1}{C_{\theta}} $, and $|L| \leqslant 3 C_{\theta }$, and suppose that $u$ verifies
\begin{itemize}
\item $u^{*} \leqslant 2$ on $Q_{1}^{*}$.
\item  $| \{ (t,x,z) \in Q_{1}^{*} : u^{*} \leqslant 0  \}  | \geqslant \frac{|Q_{1}^{*}|}{2}$.
\end{itemize}
Now, let us define $K_{+} \in \mathbb{N}^{+}$ to be the largest nonnegative integer for which $K_{+} \leqslant 1 + \frac{1}{\delta_{1}}$.  We then define a list of functions $w_{k}$, for $1\leqslant k \leqslant K_{+}$ by
\[
w_{k} =  2(w_{k-1} - 1),\quad 1\leqslant k \leqslant K_{+},\quad\mbox{with}\quad w_{0} = u.
\]
Then for every $1\leqslant k \leqslant K_{+}$ we have
\begin{equation*}
w_{k} = 2^{k} \{ u - 2  \} + 2 = 2^{k} \beta \{ \theta - L - \frac{2}{\beta } + \frac{1}{2^{k-1}\beta }  \}. 
\end{equation*}
Now, it is easy to see that for each $1\leqslant k \leqslant K_{+}$, $u^{*} \leqslant 2$ on $Q_{1}^{*}$ implies $ w_{k}^{*} = 2^{k} \{ u^{*} - 2  \} + 2 \leqslant 2 $ on $Q_{1}^{*}$. Moreover, since $\{ (t,x,z) \in Q_{1}^{*} : u^{*} \leqslant 0  \}$ is always a subset of $ \{ (t,x,z) \in Q_{1}^{*} : w_{k}^{*} \leqslant 0  \}   $ , we always have $ |\{ (t,x,z) \in Q_{1}^{*} : w_{k}^{*} \leqslant 0  \}| \geqslant \frac{|Q_{1}^{*}|}{2}  $.\\
Besides these, we also have to make the crucial observation that, for every $1\leqslant k \leqslant K_{+}$, we have
\begin{itemize}
\item  $ |2^{k} \beta | \geqslant |\beta | \geqslant \frac{1}{C_{\theta }}  $, 
and $ | L + \frac{2}{\beta } - \frac{1}{2^{k-1} \beta }  | \leqslant 6 C_{\theta } $.
\end{itemize}
This means that we can apply Lemma~\ref{Lemma1} and Lemma~\ref{Lemma2} to $w_{k}$ if we find that such an application is needed. At this point, we need to separate our discussion into two cases in the following way.\\
\emph{First}, if it happens that, for every $1\leqslant k \leqslant K_{+}$, we have $|\{ (t,x,z) \in Q_{1}^{*} : 0 < w_{k}^{*} < 1  \}| \geqslant \delta_{1}$, we then observe that we must have
\begin{equation*}
\begin{split}
|\{ (t,x,z) \in Q_{1}^{*} : w_{k}^{*} \leqslant 0  \}|
&=  |\{ (t,x,z) \in Q_{1}^{*} : w_{k-1}^{*} \leqslant 1  \} | \\ 
&\geqslant \delta_{1} +  |\{ (t,x,z) \in Q_{1}^{*} : w_{k-1}^{*} \leqslant 0  \}| ,  
\end{split}
\end{equation*}
for every $1\leqslant k \leqslant K_{+}$. Because of the above estimate, we can deduce inductively that $
|\{ (t,x,z) \in Q_{1}^{*} : w_{K_{+}}^{*} \leqslant 0  \}| \geqslant K_{+}\delta_{1} \geqslant |Q_{1}^{*}| $, which in turn tells us that $w_{K_{+}}^{*} = 2^{K_{+}} \{u^{*} - 2\} + 2 \leqslant 0$ almost everywhere on $Q_{1}^{*}$. Hence we have 
\begin{itemize}
\item $u^{*} \leqslant 2- \frac{2}{2^{K_{+}}}$, almost everywhere on $Q_{1}^{*}$.
\end{itemize}
So, we are done in the first case.\\
\indent
\emph{Second}, let us suppose the case in which there exists some $k_{0}\in \mathbb{N}$ with $1\leqslant k \leqslant K_{+}$, such that $|\{ (t,x,z) \in Q_{1}^{*} : 0 < w_{k_{0}}^{*} < 1  \}| < \delta_{1}$.  We can then apply Lemma~\ref{Lemma2} to $w_{k_{0}}$ and deduce
\begin{equation*}
\int_{Q_{\frac{1}{4}}^{*}} (w_{k_{0}}^{*} - 1)_{+}^{2} + \int_{Q_{\frac{1}{4}}} (w_{k_{0}} -1 )_{+}^{2} \leqslant D_{\theta} (\epsilon_{1})^{\frac{1}{2}} , 
\end{equation*}
which simply means 
\begin{equation*}
\int_{Q_{\frac{1}{4}}^{*}} (w_{k_{0}+1}^{*})_{+}^{2} + \int_{Q_{\frac{1}{4}}} (w_{k_{0}+1})_{+}^{2} \leqslant 4 D_{\theta} (\epsilon_{1})^{\frac{1}{2}} < \epsilon_{0} . 
\end{equation*}

Now, the above inequality tells us that we can apply Lemma~\ref{Lemma1} directly to $w_{k_{0}+1}$ over $Q_{\frac{1}{4}}^{*}$, and deduce that $w_{k_{0}+1} \leqslant 2-\lambda $ on $Q_{\frac{1}{16}}$, which implies
\begin{itemize}
\item $ u \leqslant 2 - \frac{\lambda}{2^{k_{0}+1}} $ on $Q_{\frac{1}{16}}$. 
\end{itemize}
To finish the argument, we consider the barrier function $b_{3} : B_{\frac{1}{16}}^{*} \rightarrow \mathbb{R}$ characterized by the following conditions
\begin{itemize}
\item $\triangle b_{3} = 0 $, on $B_{\frac{1}{16}}^{*}  $ .
\item $b_{3} =2$ on all the sides of $B_{\frac{1}{16}}^{*}$ except the one for $z=0$.
\item $b_{3} =  2 - \frac{\lambda}{2^{k_{0}+1}} $, on the side for $z=0$.
\end{itemize}
Then, by a simple application of the maximum principle, we know that there exists some constant $\lambda^{*}$, with $0 < \lambda^{*} < \frac{1}{2^{K_{0}+1}} \min\{1, \lambda \}$, such that $b_{3} \leqslant 2-\lambda^{*}$ on $B_{\frac{1}{32}}^{*}$. Since $u* = \beta \{ \theta^{*} - L  \}$ is harmonic and that $u^{*}$ is bounded above by $b_{3}$ along the sides of the cube $B_{\frac{1}{16}}^{*}$, it must follow that $u^{*} \leqslant b_{3} \leqslant 2- \lambda^{*}$ on $B_{\frac{1}{32}}^{*}$. So we are done in the second case.
\subsection{Proof of Lemma~\ref{Lemma1}}
The proof closely follows \cite{Driftdiffusion} except when the local energy inequality is employed.  Also we provide more details in Step Two below (step 7 in \cite{Driftdiffusion}).  For convenience, the following proof is given in the setting in which the $L^{\infty}$ solution $\theta$ of the $N$-dimensional critical Burgers' equation is defined on $[-4,0]\times \mathbb{R}^{N}$. The desired conclusion of Lemma~\ref{Lemma1} can be obtained by rescaling.\\\\
\noindent{\bf Step One: Determination of the constant $\lambda$ and of the sequence of truncated energy terms $A_{k}$.}\\
\indent We begin by constructing the universal constant $\lambda$. For this purpose, we consider the barrier function $b_{1} : B_{4}^{*} \rightarrow \mathbb{R}$ which verifies the following conditions
\begin{itemize}
\item $\triangle b_{1} = 0$ on $B_{4}^{*}$.
\item $b_{1} = 2$, on all the sides of the cube $B_{4}^{*}$, \emph{except} for the one with $z=0$.
\item $b_{1} = 0$, on the side of $B_{4}^{*}$  specified by $z=0$. 
\end{itemize}
Since $b_{1}$ is harmonic on $B_{4}^{*}$, we use the maximum principle to deduce that there exists some sufficiently small $\lambda$ with $0 < \lambda < \frac{1}{2}$, such that $0 \leqslant b_{1} \leqslant 2- 4 \lambda$ is valid over $B_{2}^{*}$. We note that $\lambda$ depends only on $N$.\\
\indent
Next, let $\theta : [-4,0]\times \mathbb{R}^{N} \rightarrow \mathbb{R}$ be a solution of \eqref{Burger}, which verifies $\|\theta \|_{L^{\infty}([-4,0]\times \mathbb{R}^{N})} \leqslant C_{\theta}$.  We set $u = \beta [\theta -L]$, with $|\beta | \geqslant  \frac{1}{C_{\theta }}$, and $|L| \leqslant 6 C_{\theta}$, and define for each $k \geqslant 1$
\[
u_{k} = \{ u - C_{k} \}_{+},\quad\mbox{and}\quad u_{k}^{*} = \{ u^{*} - C_{k} \}_{+},
\]
where $C_{k} = 2-\lambda (1 + \frac{1}{2^{k}})$.  We now consider the following quantity for each $k \geqslant 1$
\begin{equation}\label{1}
A_{k} = \int_{T_{k}}^{0}\int_{0}^{\delta^{k}}\int_{\mathbb{R}^{N}} |\nabla (\eta_{k} u_{k}^{*})|^{2} dx dz dt  + \|\eta_{k} u_{k}\|_{L^{\infty}(T_{k},0;L^{2}(\mathbb{R}^{N}))}^{2},
\end{equation}
where $T_{k} = -1-\frac{1}{2^{k}}$, and $\{\eta_{k}\}_{k=1}^{\infty}$ is a (fixed) sequence of functions in $C_{c}^{\infty} (\mathbb{R}^{N})$ such that
\[
\chi_{B(1+\frac{1}{2^{k+\frac{1}{2}}})} \leqslant \eta_{k} \leqslant \chi_{B(1+ \frac{1}{2^{k}})},
\quad\mbox{and}\quad |\nabla \eta_{k}|\leqslant C 2^{k},\quad k \geqslant 0.
\]
The integral along the $z$-direction in \eqref{1} is taken over $[0, \delta^{k} ]$, for some sufficiently small $\delta$. We will select such $\delta$, in a way depending \emph{only} on $\lambda$. We choose $\delta$ in Step Four.\\
\indent
Now, we observe that the conclusion of Lemma~\ref{Lemma1} follows at once, \emph{provided} we succeed in building up a nonlinear recurrence relation on $A_{k}$ by using the De-Giorgi's technique. We are now going to build up such a nonlinear recurrence relation for $A_{k}$ \emph{under the assumption that the following two conditions are valid}
\begin{align}
&\eta_{k}u_{k}^{*} = 0 , \forall z\in [\delta^{k} , 2  ],\label{3}\\
&\eta_{k+1} u_{k+1}^{*} \leqslant [(\eta_{k} u_{k})*P(z)]\eta_{k+1}, \forall (x,z) \in B(1+\frac{1}{2^{k}})\times [0,\delta^{k}].
\label{4}
\end{align}
The symbol $P(z)$ appearing in condition \eqref{4} stands for the Poisson kernel $P(\cdot , z)$.\\

\noindent{\bf Step Two: Establishing the nonlinear recurrence relation for $A_{k}$ by assuming the validity of conditions ~\eqref{3}   and \eqref{4} .}

To begin, we observe that for each $k \geqslant 1$, we may express the function $u-C_{k}$ as
\begin{equation*}
u-C_{k} = \beta \{  \theta -L -\frac{C_{k}}{\beta}  \} ,
\end{equation*}
with $|\beta | \geqslant \frac{1}{C_{\theta }}$, and $|L + \frac{C_{k}}{\beta}| \leqslant \{ 6 + C_{k}\}C_{\theta} \leqslant 8C_{\theta }$. This means that we can apply Proposition~\ref{LEI} directly to $u-C_{k} = \beta \{  \theta -L -\frac{C_{k}}{\beta}  \}  $, and deduce that the following inequality is valid for every $k\geqslant 0$
\begin{equation}\label{undercontrol}
\begin{split}
&\int_{\sigma}^{t} \int_{B_{2}^{*}} |\nabla (\eta u_{k}^{*})|^{2} dx dz ds + \int_{B_{2}} (\eta u_{k})^{2}(t,x) dx \\
&\quad\leqslant \int_{B_{2}} (\eta u_{k})^{2}(\sigma , x) dx + \Phi \int_{\sigma}^{t} \int_{B_{2}} |\nabla \eta|^{2} u_{k}^{2} dx ds+ \int_{\sigma}^{t}\int_{B_{2}^{*}} |\nabla \eta|^{2} (u_{k}^{*})^{2} dx dz ds,
\end{split}
\end{equation}
where, in the above inequality, we have $\Phi = 2NC_{N}\{ 8C_{\theta} + C_{\theta} \}^{2}$ and $\eta$ is some smooth cut off function compactly supported inside $B_{2}^{*}$.\\
By \emph{assuming the validity of condition \eqref{3} at step $k$}, that is
\[ 
\eta_{k} u_{k}^{*} = 0,
\]
for all $z \in [\delta^{k}, 2]$, we know that the function 
\[
\eta_{k} u_{k}^{*} \chi_{\{   0\leqslant z \leqslant \delta^{k} \}  } = 
\eta_{k} u_{k}^{*} \chi_{ \{  0\leqslant z \leqslant 2  \} }
\]
has no jump-discontinuity at $z=\delta^{k}$. We now choose some smooth function $\psi : [0,\infty )\rightarrow \mathbb{R}$ which verifies the following conditions
\begin{itemize}
\item $0\leqslant \psi (z) \leqslant 1$ , for all $z\in [0,\infty )$.
\item $\psi (z) = 1 $, for all $z \in [0,1]$.
\item $\psi (z) = 0$, for all $z \in [2,\infty )$.
\item $|\frac{d\psi}{dz}| \leqslant 2$, for all $z \in [0,\infty )$.
\end{itemize}
We then apply inequality \eqref{undercontrol} with the cut off function $\eta_{k} \psi$ and deduce that the following inequality is valid for all $\sigma$, $t$ with $T_{k-1} \leqslant \sigma \leqslant T_{k} \leqslant t \leqslant 0$ (where $T_{k} = -1 -\frac{1}{2^{k}}$)
\begin{equation}\label{C}
\begin{split}
&\int_{\sigma}^{t} \int_{B_{2}^{*}} |\nabla (\eta_{k} \psi u_{k}^{*})|^{2}  + \int_{B_{2}} (\eta_{k} u_{k})^{2}(t,x) dx \\
&\quad\leqslant \int_{B_{2}} (\eta_{k} u_{k})^{2}(\sigma , x) dx + \Phi \int_{\sigma}^{t} \int_{B_{2}} |\nabla \eta_{k}|^{2} u_{k}^{2} dx ds+ \int_{\sigma}^{t}\int_{B_{2}^{*}} |\nabla (\eta_{k} \psi )|^{2} (u_{k}^{*})^{2} dx dz ds.
\end{split}
\end{equation}
We next notice that $ \eta_{k} \psi u_{k}^{*} = \eta_{k} u_{k}^{*} \chi_{\{ 0\leqslant z \leqslant \delta^{k}\} } = \eta_{k} u_{k}^{*} \chi_{\{   0\leqslant z \leqslant 2 \}} $, and this implies that 
\begin{equation*}
\int_{\sigma}^{t} \int_{B_{2}^{*}} |\nabla (\eta_{k} \psi u_{k}^{*})|^{2} = 
\int_{\sigma}^{t} \int_{0}^{2}\int_{\mathbb{R}^{N}} |\nabla ( \eta_{k} u_{k}^{*} \chi_{\{   0\leqslant z \leqslant 2 \}}    )|^{2} 
= \int_{\sigma}^{t} \int_{0}^{\delta^{k}}\int_{\mathbb{R}^{N}} |\nabla ( \eta_{k} u_{k}^{*}     )|^{2} 
\end{equation*}
Next, let us recall that according to the definition of $\eta_{k}$ we have
\begin{itemize}
\item $ \eta_{k} \leqslant \chi_{B(1+\frac{1}{2^{k}})} \leqslant \chi_{B(1+\frac{1}{2^{k-\frac{1}{2}}   })} \leqslant \eta_{k-1} \leqslant  \chi_{B(1+\frac{1}{2^{k-1}   })}                 $ . 
\item $|\nabla \eta_{k}| \leqslant C 2^{k} \chi_{B(1+\frac{1}{2^{k}})}$.
\end{itemize}

So, it follows that

\begin{itemize}
\item $|\nabla ( \eta_{k} \psi   )| \leqslant |\nabla \eta_{k}| \psi + \eta_{k} |\frac{d \psi}{dz}|
\leqslant (C 2^{k} +2) \chi_{B(1+\frac{1}{2^{k}})} \leqslant (C 2^{k} +2) \eta_{k-1}$ . 
\item $ |\nabla \eta_{k}| \leqslant C 2^{k} \chi_{B(1+\frac{1}{2^{k}})} \leqslant C 2^{k} \eta_{k-1}    $.
\end{itemize}
Combining all these, it follows from \eqref{C} that the following inequality is valid for all $\sigma$, $t$ with $T_{k-1} \leqslant \sigma \leqslant T_{k} \leqslant t \leqslant 0$ 
\begin{equation*}
\begin{split}
&\int_{\sigma}^{t} \int_{0}^{\delta^{k}} \int_{\mathbb{R}^{N}}        |\nabla (\eta_{k}  u_{k}^{*})|^{2}  + \int_{\mathbb{R}^{N}} (\eta_{k} u_{k})^{2}(t,x) dx \\
&\quad\leqslant \int_{\mathbb{R}^{N}} (\eta_{k} u_{k})^{2}(\sigma , x) dx 
 + \Phi \int_{\sigma}^{t} \int_{\mathbb{R}^{N}} C 2^{2k}(\eta_{k-1} u_{k})^{2} dx ds
+ \int_{\sigma}^{t}\int_{0}^{2} \int_{\mathbb{R}^{N}}
C   2^{2k}       ( \eta_{k-1}   u_{k}^{*})^{2} dx dz ds.
\end{split}
\end{equation*}
By taking the average among all the terms appearing in the above inequality over the variable $\sigma \in [T_{k-1}, T_{k}]$, and then taking the $\sup$ over $t\in [T_{k}, 0 ]$, we yield the following
\begin{equation*}
\begin{split}
&A_{k} \leqslant 2^{k} \int_{T_{k-1}}^{T_{k}} \int_{\mathbb{R}^{N}} (\eta_{k} u_{k})^{2}(\sigma , x) dx d\sigma\\ 
&\quad\quad\quad+  C(1+\Phi )2^{2k}
\{  \int_{T_{k-1}}^{0} \int_{\mathbb{R}^{N}} (\eta_{k-1} u_{k})^{2} dx ds
+ \int_{T_{k-1}}^{0}\int_{0}^{2} \int_{\mathbb{R}^{N}}
  ( \eta_{k-1}   u_{k}^{*})^{2} dx dz ds.\}
\end{split}
\end{equation*}
Our goal is to raise up the index for the three terms appearing in the right hand side of the above inequality. We just focus on $ \int_{T_{k-1}}^{0}\int_{0}^{2} \int_{\mathbb{R}^{N}}   ( \eta_{k-1}   u_{k}^{*})^{2}   $ (which is the most difficult among the three), and we remember our lucky number $2(\frac{N+1}{N})$ from the process of applying the De-Giorgi's method in section \ref{weaksolutions}.\\
\indent Now, by using the facts $\chi_{\{ u_{k}^{*} > 0  \}} \leqslant \chi_{\{ u_{k-1}^{*} > \frac{\lambda }{2^{k}}   \}} \leqslant \frac{2^{k}}{\lambda } u_{k-1}^{*}$, and 
$\eta_{k-1} \leqslant \chi_{B(1+\frac{1}{2^{k-1}})} \leqslant \eta_{k-2}$, we can deduce that
\begin{equation*}
\begin{split}
 ( \eta_{k-1}   u_{k}^{*})^{2} 
\leqslant \chi_{B(1+\frac{1}{2^{k-1}})}(u_{k}^{*})^{2} \chi_{\{ u_{k}^{*} > 0  \}} 
& \leqslant \chi_{B(1+\frac{1}{2^{k-1}})}(u_{k}^{*})^{2} ( \frac{2^{k}}{\lambda }u_{k-1}^{*})^{\frac{2}{N}} \\
& \leqslant \chi_{B(1+\frac{1}{2^{k-1}})} ( \frac{2^{k}}{\lambda } )^{\frac{2}{N}}
( u_{k-1}^{*}  )^{2(\frac{N+1}{N})} \\
& \leqslant ( \frac{2^{k}}{\lambda } )^{\frac{2}{N}}
( \eta_{k-2} u_{k-1}^{*}  )^{2(\frac{N+1}{N})}
\end{split}
\end{equation*}
Hence, it follows at once from the above inequality that
\begin{equation}\label{D}
\begin{split}
\int_{T_{k-1}}^{0}\int_{0}^{2} \int_{\mathbb{R}^{N}}
( \eta_{k-1}   u_{k}^{*})^{2} dx dz ds
& \leqslant \int_{T_{k-1}}^{0}\int_{0}^{2} \int_{\mathbb{R}^{N}} ( \frac{2^{k}}{\lambda } )^{\frac{2}{N}}
( \eta_{k-2} u_{k-1}^{*}  )^{2(\frac{N+1}{N})} \\
& \leqslant   ( \frac{2^{k}}{\lambda } )^{\frac{2}{N}}\int_{T_{k-1}}^{0}\int_{0}^{2} 
\|\eta_{k-2} u_{k-2}^{*}\|_{L^{2(\frac{N+1}{N})}(\mathbb{R}^{N})}^{2(\frac{N+1}{N})} dz dt \\
&  = ( \frac{2^{k}}{\lambda } )^{\frac{2}{N}}\int_{T_{k-1}}^{0}\int_{0}^{\delta^{k-2}} 
\|\eta_{k-2} u_{k-2}^{*}\|_{L^{2(\frac{N+1}{N})}(\mathbb{R}^{N})}^{2(\frac{N+1}{N})} dz dt
\end{split}
\end{equation}
(In the last line of the above estimate, we implicitly employ \eqref{3} at step $k-2$.)
Now, by \emph{assuming the validity of \eqref{3} at step $k-3$} ( That is, 
$\eta_{k-3} u_{k-3}^{*} = 0$, for all $z \in [\delta^{k-3}, 2]$ ), we know that the function $\eta_{k-3} u_{k-3}^{*} \chi_{\{   0\leqslant z \leqslant \delta^{k-3} \}  } = 
\eta_{k-3} u_{k-3}^{*} \chi_{ \{  0\leqslant z \leqslant 2  \} }$ has no jump-discontinuity at $z=\delta^{k-3}$ and has the same trace as $(\eta_{k-3} u_{k-3})^{*}  $ at $z=0$, we can use the energy minimization property of harmonic extension to deduce that the following estimate is valid at step $k-3$.

\begin{equation*}
\begin{split}
\int_{0}^{\delta^{k-3}} \int_{\mathbb{R}^{N}} |\nabla ( \eta_{k-3} u_{k-3}^{*}  )|^{2} dx dz
& = \int_{0}^{\infty} \int_{\mathbb{R}^{N}} |\nabla ( \eta_{k-3} u_{k-3}^{*} \chi_{\{   0\leqslant z \leqslant \delta^{k-3} \}}  )|^{2} dx dz \\
& \geqslant \int_{0}^{\infty} \int_{\mathbb{R}^{N}} |\nabla \{ (\eta_{k-3} u_{k-3})^{*} \} |^{2} dx dz \\
& = \int_{\mathbb{R}^{N}}  \eta_{k-3} u_{k-3} \cdot (-\triangle )^{\frac{1}{2}} ( \eta_{k-3} u_{k-3}  )  dx .
\end{split}
\end{equation*}
Because of this last inequality, we can use the Sobolev embedding and H\"older's inequality to obtain
\begin{equation*}
\|\eta_{k-3} u_{k-3}\|_{L^{2(\frac{N+1}{N})}( [T_{k-3} , 0]\times \mathbb{R}^{N})} \leqslant A_{k-3}^{\frac{1}{2}} . 
\end{equation*}
By comparing the above inequalities, we see that we need a passage from the term $ \|\eta_{k-2} u_{k-2}^{*}\|_{L^{2(\frac{N+1}{N})}(\mathbb{R}^{N})}  $ to the term $\|\eta_{k-3} u_{k-3}\|_{L^{2(\frac{N+1}{N})}(\mathbb{R}^{N})}  $, and such a passage is provided to us by condition \eqref{4}  (at step $k-3$ ). Indeed, \emph{ by assuming the validity of condition \eqref{4} at step $k-3,$} Young's inequality tells us that, for every $t\in [-2,0]$ and every $z\in (0,\delta^{k-2} )$, we have
\begin{equation*}
\begin{split}
\|\eta_{k-2} u_{k-2}^{*}\|_{L^{2(\frac{N+1}{N})}(\mathbb{R}^{N})}
& \leqslant \|P(z)\|_{L^{1}(\mathbb{R}^{N})} \|\eta_{k-3} u_{k-3}\|_{L^{2(\frac{N+1}{N})}(\mathbb{R}^{N})} \\
& = \|P(1)\|_{L^{1}(\mathbb{R}^{N})} \|\eta_{k-3} u_{k-3}\|_{L^{2(\frac{N+1}{N})}(\mathbb{R}^{N})} ,
\end{split}
\end{equation*}
where the last equality is valid just because we always have $  \|P(z)\|_{L^{1}(\mathbb{R}^{N})}  = \|P(1)\|_{L^{1}(\mathbb{R}^{N})}$. So, it follows from \eqref{D} that 
\begin{equation*}
\begin{split}
\int_{T_{k-1}}^{0}\int_{0}^{2} \int_{\mathbb{R}^{N}} ( \eta_{k-1}   u_{k}^{*})^{2} dx dz ds 
& \leqslant ( \frac{2^{k}}{\lambda } )^{\frac{2}{N}}\int_{T_{k-1}}^{0}\int_{0}^{\delta^{k-2}} 
\|\eta_{k-2} u_{k-2}^{*}\|_{L^{2(\frac{N+1}{N})}(\mathbb{R}^{N})}^{2(\frac{N+1}{N})} dz dt \\
& \leqslant ( \frac{2^{k}}{\lambda } )^{\frac{2}{N}}\int_{T_{k-1}}^{0}\int_{0}^{\delta^{k-2}} 
\|P(1)\|_{L^{1}(\mathbb{R}^{N})}^{2(\frac{N+1}{N})} \|\eta_{k-3} u_{k-3}\|_{L^{2(\frac{N+1}{N})}(\mathbb{R}^{N})}^{2(\frac{N+1}{N})} dz dt \\
& \leqslant \delta^{k-2} \|P(1)\|_{L^{1}(\mathbb{R}^{N})}^{2(\frac{N+1}{N})}( \frac{2^{k}}{\lambda } )^{\frac{2}{N}} \|\eta_{k-3} u_{k-3}\|_{L^{2(\frac{N+1}{N})}([T_{k-3}, 0]\times \mathbb{R}^{N})}^{2(\frac{N+1}{N})} \\
& \leqslant  \|P(1)\|_{L^{1}(\mathbb{R}^{N})}^{2(\frac{N+1}{N})}( \frac{2^{k}}{\lambda } )^{\frac{2}{N}} A_{k-3}^{1+\frac{1}{N}} .
\end{split}
\end{equation*}
So, we have raised up the index for $\int_{T_{k-1}}^{0}\int_{0}^{2} \int_{\mathbb{R}^{N}} ( \eta_{k-1}   u_{k}^{*})^{2} dx dz ds   $. The other two terms, namely 
$ \int_{T_{k-1}}^{T_{k}} \int_{\mathbb{R}^{N}} (\eta_{k} u_{k})^{2}(\sigma , x) dx d\sigma  $ 
and $ \int_{T_{k-1}}^{0} \int_{\mathbb{R}^{N}} (\eta_{k-1} u_{k})^{2} dx ds   $ can be treated in a similar way. As a result, with the assistance of condition \eqref{3} and condition \eqref{4}, we are able to obtain the following nonlinear recurrence relation at step $k$.
\begin{equation}\label{2}
A_{k} \leqslant C_{0}^{k}A_{k-3}^{1+\frac{1}{N}} , 
\end{equation}
where, in the above nonlinear recurrence relation, $C_{0}$ stands for some constant depending only on $C_{\theta}$ and $N$. More precisely, we can summarize what we have done in the following way
\begin{itemize}
\item For every $k \geqslant 3$, if condition \eqref{3} is valid at steps $k-3$, $k-2$, $k$, and condition \eqref{4} is valid at step $k-3$, then it follows that the nonlinear recurrence relation \eqref{2} is valid at step $k$ also. 
\end{itemize}

\noindent{\bf  Step Three: Establishing condition \eqref{4} at step $k$ by assuming the validity of condition \eqref{3} at step $k$ . }\\
We need to introduce another barrier function $b_{2} : [0,\infty ) \times [0,1] \rightarrow \mathbb{R}$, which verifies  
\begin{itemize}
\item $\triangle b_{2}= 0 $ on  $(0,\infty ) \times (0,1)$.
\item $b_{2} (0,z) = 2$ , for $z \in (0,1)$.
\item $ b_{2}(x,0) = b_{2}(x,1) = 0 $, for $x\in (0,\infty )   $.
\end{itemize}
Now, \emph{by assuming the validity of \eqref{3} at step $k$,} we are ready to establish condition \eqref{4} at step $k$ by controlling the behavior of $u_{k}^{*}$ over $B(1+\frac{1}{2^{k+\frac{1}{2}}})\times [0, \delta^{k}]$, where the suitable $\delta$ will be chosen (once and for all, and in a way depending only on $N$) during this procedure.

Indeed, a direct application of the maximum principle (together with \eqref{3} at step $k$ ) yields the following expression on $B(1+\frac{1}{2^{k+\frac{1}{2}}})\times [0, \delta^{k}]$
\begin{equation*}
u_{k}^{*} \leqslant (\eta_{k} u_{k})*P(z) + 
\sum_{i=1}^{N} b_{2}( \frac{x^{+} -x_{i} }{\delta^{k}}  ,  \frac{z}{\delta^{k}}  ) + 
b_{2}( \frac{ x_{i} - x^{-}    }{\delta^{k}}  ,  \frac{z}{\delta^{k}}  )  ,
\end{equation*}
where $x^{+} =(1+\frac{1}{2^{k+\frac{1}{2}}}) $, $x^{-} = -x^{+} $ , and that $b_{2}$ verifies $b_{2} (x,z) \leqslant 2(2)^{\frac{1}{2}} e^{\frac{-x}{2}}$. Hence, when the scaling factor $\frac{1}{\delta^{k}}$ gets involved in the variables of $b_{2}$, we yield the following inequality which is valid over the same set $B(1+\frac{1}{2^{k+\frac{1}{2}}})\times [0, \delta^{k}]$.

\begin{equation*}
b_{2}( \frac{x^{+} -x_{i} }{\delta^{k}}  ,  \frac{z}{\delta^{k}}  ) + 
b_{2}( \frac{ x_{i} - x^{-}    }{\delta^{k}}  ,  \frac{z}{\delta^{k}}  ) 
\leqslant 2(2)^{\frac{1}{2}} [e^{\frac{-(x^{+}-x_{i})}{2 \delta^{k}}} +  e^{\frac{-(x_{i}-x^{-})}{2 \delta^{k}}}   ] .
\end{equation*}

Now, if we restrict $x$ to be in $B(1+\frac{1}{2^{k+1}})$, we have  
$\min \{ |x^{+}- x_{i}| , |x_{i} - x^{-} | \} \geqslant \frac{1}{2(2^{\frac{1}{2}} + 1 )2^{k} }$, for all $x \in B(1+\frac{1}{2^{k+1}})$. So, the above inequality implies that the following holds over the smaller set $B(1+\frac{1}{2^{k+1}})\times [0,\delta^{k}]$

\begin{equation*}
 b_{2}( \frac{x^{+} -x_{i} }{\delta^{k}}  ,  \frac{z}{\delta^{k}}  ) + 
b_{2}( \frac{ x_{i} - x^{-}    }{\delta^{k}}  ,  \frac{z}{\delta^{k}}  )
\leqslant 2 (2 (2^{\frac{1}{2}})) e^{ \frac{ -1 }{4( 2^{\frac{1}{2}} +1  ) 2^{k} \delta^{k}  }  } .
\end{equation*}
Hence, the following inequality is also valid over $B(1+\frac{1}{2^{k+1}})\times [0,\delta^{k}]$
\begin{equation*}
u_{k}^{*} \leqslant (\eta_{k} u_{k})*P(z) + 2N(2 (2^{\frac{1}{2}})) e^{ \frac{ -1 }{4( 2^{\frac{1}{2}} +1  ) 2^{k} \delta^{k}  }  } .
\end{equation*}
Since $u_{k+1}^{*} \leqslant [u_{k}^{*} -\frac{\lambda }{2^{k+1}}]_{+}$, the above inequality implies that we have over  $B(1+\frac{1}{2^{k+1}})\times [0,\delta^{k}]$
\begin{equation}\label{8}
u_{k+1}^{*} \leqslant [   (\eta_{k} u_{k})*P(z) + 2N(2 (2^{\frac{1}{2}})) e^{ \frac{ -1 }{4( 2^{\frac{1}{2}} +1  ) 2^{k} \delta^{k}  }  } - \frac{\lambda }{2^{k+1}}  ]_{+} .
\end{equation}
Here, let us discuss what we have done. The above way of arriving at inequality \eqref{8} is just a simple application of the maximum principle with the participation of the barrier function $b_{2}$ with its width in the z-direction being compressed by the scaling factor $\frac{1}{\delta^{k}}$. However, \eqref{8} eventually forces us to compare $ 2N(2 (2^{\frac{1}{2}})) e^{ \frac{ -1 }{4( 2^{\frac{1}{2}} +1  ) 2^{k} \delta^{k}  }  }$ with $\frac{\lambda }{2^{k+1}}$. This motivates us to choose $\delta$ to be sufficiently small \emph{so that the following holds for all $k\geqslant 1$}
\begin{equation}\label{5}
2N(2 (2)^{\frac{1}{2}}) e^{-\frac{1}{4 \{ (2)^{\frac{1}{2}} + 1 \}  (2\delta )^{k}  }} \leqslant \frac{\lambda}{2^{k+2}} . 
\end{equation}
Observe that $\delta$, which makes \eqref{5} valid for all $k \geqslant 1$, depends only on $N$.  Once $\delta$ is chosen and fixed, \eqref{8} (at step $k$) together with the assistance of \eqref{5} give $u_{k+1}^{*} \leqslant [(\eta_{k} u_{k})*P(z) - \frac{\lambda}{2^{k+2}}   ]_{+}$ over $B(1+\frac{1}{2^{k+1}})\times [0,\delta^{k}]$, and this in turn gives us the validity of condition \eqref{4} at step $k$. Now, let us summarize what we have achieved in this step
\begin{itemize}
\item  We can always select a sufficiently small $\delta > 0$ for which \emph{condition~\ref{5} is valid for all $k \geqslant 1$}. For such $\delta > 0$, the validity of condition \eqref{3} at step $k$ directly implies the validity of condition \eqref{4} at step $k$.  
\end{itemize}
\noindent{\bf Step Four : Propagation of condition \eqref{3}.}\\
Now, let $\delta > 0$ be the fixed, sufficiently small constant which makes condition \eqref{5} valid for all $k \geqslant 1$. Now, we \emph{attempt} to derive condition \eqref{3} \emph{at step $k+1$} by \emph{assuming} the validity of \eqref{3} at step $k$.

To do this, let us recall that inequality \eqref{8} at step $k$ and condition~\ref{5} together give $u_{k+1}^{*} \leqslant [(\eta_{k} u_{k})*P(z) - \frac{\lambda}{2^{k+2}}   ]_{+}$ over $B(1+\frac{1}{2^{k+1}})\times [0,\delta^{k}]$. In order to obtain \eqref{3} at step $k+1$, we may just take advantage of the inequality we just mentioned and deduce that
\begin{equation*}
\|(\eta_{k} u_{k}) * P(z)\|_{L^{\infty}(\mathbb{R}^{N})} 
\leqslant \|\eta_{k} u_{k}\|_{L^{2}(\mathbb{R}^{N})} \|P(z)\|_{L^{2}(\mathbb{R}^{N})}
\leqslant (A_{k})^{\frac{1}{2}} \frac{1}{z^{\frac{N}{2}}} \|P(1)\|_{L^{2}(\mathbb{R}^{N})},
\end{equation*}
where the second inequality comes from the definition of $A_{k}$ and the fact that $\|P(z)\|_{L^{2}(\mathbb{R}^{N})} = \frac{1}{z^{\frac{N}{2}}} \|P(1)\|_{L^{2}(\mathbb{R}^{N})}   $. 
But this eventually tells us that over $B(1+\frac{1}{2^{k+1}})\times [\delta^{k+1}, \delta^{k}]$ we have
\begin{equation}\label{9}
u_{k+1}^{*} \leqslant [  (A_{k})^{\frac{1}{2}} \frac{1}{z^{\frac{N}{2}}} \|P(1)\|_{L^{2}(\mathbb{R}^{N})}   -\frac{\lambda}{2^{k+2}} ]_{+}
\leqslant [  (A_{k})^{\frac{1}{2}}  \frac{1}{\delta^{\frac{N}{2}(k+1)}}  \|P(1)\|_{L^{2}(\mathbb{R}^{N})} -  \frac{\lambda}{2^{k+2}}  ]_{+}.
\end{equation}
(We note that the second inequality is valid because $\delta^{k+1} \leqslant z \leqslant \delta^{k}$). Here, we have to keep in mind that condition \eqref{3} \emph{at step $k+1$} is what we want. So, by taking a closer look at \eqref{9}, it is natural that we want to have the following inequality (because we want $u_{k+1}^{*} = 0$ on $B(1+\frac{1}{2^{k+1}})\times [\delta^{k+1}, \delta^{k}]$ )
\begin{equation}\label{10}
(A_{k})^{\frac{1}{2}}  \frac{1}{\delta^{\frac{N}{2}(k+1)}}  \|P(1)\|_{L^{2}(\mathbb{R}^{N})} \leqslant \frac{\lambda}{2^{k+2}} .
\end{equation}
But this simply forces us to admit that the following two conditions should be true for \emph{any} sufficiently large $M$ ($M$ should be greater than $\{\frac{2}{\delta^{\frac{N}{2}}}\}^{2} $)
\begin{equation}\label{6}
A_{k} \leqslant \frac{1}{M^{k}} .
\end{equation}
\begin{equation}\label{7}
\frac{1}{M^{\frac{k}{2}}\delta^{\frac{N}{2}(k+1)}} \|P(1)\|_{L^{2}(\mathbb{R}^{N})} \leqslant \frac{\lambda }{2^{k+2}} , \forall k \geqslant 1. 
\end{equation}
The reason is that we have already seen a sequence $\{  ( \frac{\delta^{\frac{N}{2}}}{2} )^{k}   \}_{k=1}^{\infty}$ appearing in inequality \eqref{10}, which is a sequence decaying to $0$ as $k$ increases. If we would like to construct another sequence decaying in a rate faster than $\{  ( \frac{\delta^{\frac{N}{2}}}{2} )^{k}   \}_{k=1}^{\infty}$, the best thing to do is to choose some $\{  \frac{1}{M^{k}}   \}_{k=1}^{\infty}$, with $M$ to be large when compared with $\frac{2}{\delta^{\frac{N}{2}}}$. Hence $\{  \frac{1}{M^{k}}   \}_{k=1}^{\infty}$ will decay faster than $\{  (\frac{\delta^{\frac{N}{2}}}{2} )^{k}   \}_{k=1}^{\infty}$. This observation more or less explains the origins of conditions \eqref{6} and \eqref{7}. However, it is important to observe that condition \eqref{7} \emph{automatically becomes valid for any sufficiently large $M > 1$}, while condition \eqref{6} and condition \eqref{3} mutually depend on each other in a delicate way, just as we will see in our next step. But now, let us summarize the result we have obtained in this step as follow.
\begin{itemize}
\item For the fixed choice of sufficiently small $\delta$ as selected in Step Three, and for any sufficiently large $M >1$ as selected in Step Four, condition \eqref{3} at step $k$, together with condition \eqref{6} at step $k$, will imply the validity of condition \eqref{3} at step $k+1$. 
\end{itemize}

\noindent{\bf Step Five : Propagation of condition \eqref{6} and it's relation to the nonlinear recurrence relation \eqref{2} for the truncated energy terms $A_{k}$.}\\
In this step, by assuming condition \eqref{3} at step $k-3$ and also condition \eqref{6} at steps $k-3$, $k-2$, $k-1$, we \emph{attempt} to deduce the validity of condition \eqref{6} at step $k$. In our present circumstance, by applying the conclusion of Step Four to conditions \eqref{3} and \eqref{6} at steps $k-3$, $k-2$, and $k-1$ successively, we can deduce that \emph{our assumptions will imply the validity of condition \eqref{3} at steps $k-2$, $k-1$, and $k$ also}. Hence, we can invoke the conclusion obtained in Step Two to deduce that $A_{k} \leqslant C_{0}^{k}A_{k-3}^{1+\frac{1}{N}}$ is valid at step $k$. This, together with the validity of condition \eqref{6} at step $k-3$, will in turns imply that we have the following inequality to be valid
\begin{equation*}
A_{k} \leqslant C_{0}^{k}A_{k-3}^{1+\frac{1}{N}} \leqslant C_{0}^{k} \{ \frac{1}{M^{k-3}} \}^{1+\frac{1}{N}} .
\end{equation*}
Because of the above inequality, we will have the validity of condition \eqref{6} at step $k$, \emph{provided if $M$ is chosen to be sufficiently large so that the following condition becomes valid for all $k \geqslant 12 N$ } (For more details about this, see Lemma 7 \cite{Driftdiffusion}).
\begin{equation}\label{closecircle}
\frac{1}{M^{k}} \geqslant  C_{0}^{k} \{ \frac{1}{M^{k-3}} \}^{1+\frac{1}{N}} .
\end{equation}
More precisely, we obtain the following conclusion in this step
\begin{itemize}
\item  If $M$ is chosen to be large enough, condition \eqref{closecircle} will become valid for all  $k \geqslant 12 N$ (For a formal proof of this fact, see Lemma 7 of \cite{Driftdiffusion} ). For any such sufficiently large $M$ being selected, condition \eqref{3} at step $k-3$, together with condition \eqref{6} at steps $k-3$, $k-2$, $k-1$, will give the validity of condition \eqref{6} at step $k$.
\end{itemize}
\noindent{\bf Step Six : Completing  the argument by taking the initial steps.}\\
Before we complete the proof of Lemma \eqref{Lemma1}, let us summarize what we have achieved from Step One to Step Five.

We recall that after the universal constant $\lambda \in ( 0 , \frac{1}{2} )$ is chosen in Step One, we have determined $\delta > 0$ (which depends only on $N$), and some sufficiently large $M > 1 $ (which depends only on $N$ and $C_{\theta }$) such that the following three conditions (which are conditions \eqref{5}, \eqref{7}, and \eqref{closecircle} respectively) are valid at the same time \emph{for all $k \geqslant 1$}.

\begin{itemize}
\item  $   2N(2 (2)^{\frac{1}{2}}) e^{-\frac{1}{4 \{ (2)^{\frac{1}{2}} + 1 \}  (2\delta )^{k}  }} \leqslant \frac{\lambda}{2^{k+2}}   $ .
\item  $    \frac{1}{M^{\frac{k}{2}}\delta^{\frac{N}{2}(k+1)}} \|P(1)\|_{L^{2}(\mathbb{R}^{N})} \leqslant \frac{\lambda }{2^{k+2}}   $ .
\item  $    \frac{1}{M^{k}} \geqslant  C_{0}^{k} \{ \frac{1}{M^{k-3}} \}^{1+\frac{1}{N}}    $ .
\end{itemize}

With the technical support of the three conditions listed as above, we have also demonstrated that the propagation of condition \eqref{3} and the propagation of condition \eqref{6} mutually rely on each other in the following way.

\begin{itemize}
\item condition \eqref{3} at step $k$, together with condition \eqref{6} at step $k$, will give condition \eqref{3} at step $k+1$. 
\item condition \eqref{3} at step $k-3$, together with condition \eqref{6} at steps $k-3$, $k-2$, $k-1$, will give condition \eqref{6} at step $k$.
\end{itemize}

Because of this, we can conclude our proof for Lemma~\ref{Lemma1} by selecting some sufficiently small $\epsilon_{0}$ (in a way depending only on $N$ and the given constant $C_{\theta }$) such that the following two statements are true (This is sufficient because the validity of condition \eqref{6} for all $k\geqslant 1$ immediately gives the desired conclusion of Lemma \ref{Lemma1} ).

\begin{itemize}
\item $A_{k} \leqslant \frac{1}{M^{k}}$, for every $0 \leqslant k \leqslant 12N$. 
\item $\eta_{0} u_{0}^{*} = 0 $, for all $z \in [1 , 2]$.
\end{itemize}

To see the way in which the $\epsilon_{0}$ is selected, we just recall that the function $u = \beta [\theta -L]$ (with $|\beta | \geqslant  \frac{1}{C_{\theta }}$, and $|L| \leqslant 6 C_{\theta}$) under consideration is required to satisfy the hypothesis that

\begin{itemize}
\item  $u^{*} = \beta [\theta^{*} - L]$ verifies $u^{*}\leqslant 2$ on $[-4,0]\times B_{4}^{*}$, and that $\int_{-4}^{0}\int_{B_{4}^{*}} (u_{+}^{*})^{2} + \int_{-4}^{0}\int_{B_{4}} u_{+}^{2} \leqslant \epsilon_{0} $ . 
\end{itemize}

So, we may invoke inequality \eqref{C}, which we obtained in Step Two, to deduce that $u^{*} = \beta [\theta^{*} - L]$ must satisfy the following inequality \emph{for every $0 \leqslant k \leqslant 12N$}.

\begin{equation*}
A_{k}  \leqslant \int_{T_{k}}^{0} \int_{B_{2}^{*}} |\nabla (\eta_{k} \psi u_{k}^{*})|^{2}  + sup_{t\in [T_{k},0]} \int_{B_{2}} (\eta_{k} u_{k})^{2}(t,x) dx
\leqslant  C 2^{24N} (1+ \Phi) \epsilon_{0} ,               
\end{equation*}

where, in the above inequality, $C$ stands for some constant depending only on $N$ and we have implicitly use the fact that $|\nabla \eta_{k}|\leqslant C 2^{k}$, and $0 \leqslant k \leqslant 12N$. 

Because of the above inequality, we know that if $\epsilon$ satisfies 

\begin{equation*}
0 < \epsilon < \{ M^{12N} 2^{24N} C ( 1 + \Phi )           \}^{-1},
\end{equation*}
then it follows at once that condition \eqref{6} is valid for $1\leqslant k \leqslant 12N$. On the other hand, we also need to control the behavior of $u^{*}$ over $B_{2}\times [1,2]$ by the upper bound $2-2\lambda$ (because this will give $  u_{0}^{*} = \{ u^{*} -  (2-2\lambda )   \}_{+} \leqslant 0  $ on $B_{2}\times [1,2]$, and hence the validity of condition \eqref{3} at step $0$ ). To achieve this, we use the local energy inequality to observe that the following estimate is valid for all $z \in [1,2]$
\begin{itemize}
\item $  \|u_{+}\chi_{B_{2}} * P(z)\|_{L^{\infty}(\mathbb{R}^{N})} \leqslant \tilde C_\theta \|P(z)\|_{L^{2}(\mathbb{R}^{N}   )} (\epsilon_{0})^{\frac{1}{2}}$ ,
\end{itemize}
where, in the above estimate, $\tilde C_\theta$ stands for some constant depending on $N$ and $C_{\theta}$.  We now recall that the barrier function $b_{1}$ as constructed in Step 1 verifies $0 \leqslant b_{1} \leqslant 2-4\lambda$ on $B_{2}^{*}$. So, a simple application of the maximum principle to $u^{*}$ will give the following bound for $u^{*}$ over $B_{2}\times [1,2]$.

\begin{itemize}
\item $u^{*} \leqslant 2 - 4 \lambda +  \tilde C_\theta \|P(1)\|_{L^{2}(\mathbb{R}^{N}   )} (\epsilon_{0})^{\frac{1}{2}}$ on   $B_{2}\times [1,2]$.
\end{itemize}
This means that if we select any $\epsilon_{0}$ with $0 < \epsilon_{0} <  \{ \frac{2\lambda }{ \tilde C_\theta \|P(1)\|_{L^{2}(\mathbb{R}^{N}   )}  } \}^{2}$, it will follow at once that $  u_{0}^{*} = \{ u^{*} -  (2-2\lambda )   \}_{+} \leqslant 0  $ on $B_{2}\times [1,2]$, and hence the validity of condition \eqref{3} at step $0$. So, finally, we conclude that the $\epsilon_{0}$ as required in Lemma~\ref{Lemma1} can be any positive number less than $min \{ \{ M^{12N} 2^{24N} C ( 1 + \Phi )           \}^{-1} ,  \{ \frac{2\lambda }{ \tilde C_\theta \|P(1)\|_{L^{2}(\mathbb{R}^{N}   )}  } \}^{2}               \}$, and we have completed the proof for Lemma~\ref{Lemma1}.

\subsection{Proof of Lemma \ref{Lemma2}}
The proof uses the following lemma, the proof of which can be found in Appendix A of \cite{Driftdiffusion}.
\begin{lem}{{\bf( De-Giorgi Isoperimetric Lemma}} \cite{Driftdiffusion})
Let $\omega \in \dot H^1([-1, 1]^{N+1})$, and
\begin{align*}
\mathcal A = \{x: \omega(x) \leq 0\},\\
\mathcal B = \{x: \omega(x) \geq 1\},\\
\mathcal C = \{x: 0 < \omega(x) < 1\},
\end{align*}
then
\[
|\mathcal A||\mathcal B|\leq C_N \|\omega\|_{\dot H^1}|\mathcal C|^\frac{1}{2}.
\]
\end{lem}
We now present the proof of Lemma \ref{Lemma2}.  The proof is \emph{exactly} the same as the proof of Lemma 8 in \cite{Driftdiffusion}.  We only make changes to its presentation, which we believe make it easier to follow. For convenience, the following proof is also given in the setting in which the $L^{\infty}$ solution $\theta$ of the $N$-dimensional critical Burgers' equation is defined on $[-4,0]\times \mathbb{R}^{N}$ and the desired conclusion of Lemma~\ref{Lemma2} can be obtained by rescaling.\newline\indent
Start with choosing $\epsilon_1 \ll 1$.  Next, since $u^*\leq 2$ on $Q_4^*$, by the local energy inequality there exists some constant $C$ such that
\be\label{makeupsomething2}
\int^0_{-4}\int_{B_1^*} |\nabla u^\ast_+|^2dxdzdt \leq C_1.
\ee
Now we make two observations.  First, since $u^* \leq 2$ on $Q^*_4$, then
\[
(u^*-1)_+ \leq 1.  
\]
Second, if we let
\[
U= \left \{ (t,x,z) \in Q^*_{1}: u^*(t,x,z) \geq 1 \right\},
\]
then
\be
\int_{Q^*_{1}}([u^*-1]_+)^2dtdxdz=\int_U([u^*-1]_+)^2dtdxdz.
\ee
It follows that if we can show that the measure of the set $U$ satisfies $|U|\leq \frac{\epsilon_1}{2}(1+\sqrt{C_1})$, then
\begin{equation}\label{makeupsomething}
 \begin{split}
\int_{Q^*_{1}}([u^*-1]_+)^2dtdxdz&=\int_U([u^*-1]_+)^2dtdxdz\\
&\leq\int_U 1 dtdxdz
\leq \frac{\epsilon_1}{2}(1+\sqrt{C_1}),
\end{split}
\end{equation}
which gives the first part of \eqref{mlemm1} for $\epsilon_1$ small enough and $D_\theta$ chosen as below in \eqref{Dthetadef}.  The second part follows from this one since following exactly \cite{Driftdiffusion}, for $t,x$ fixed
\be
u_{+}(t,x)=u_{+}^*(t,x,z)-\int^z_0\partial_zu_{+}^*(t,x,z) d\bar z,
\ee
and
\be
(u-1)^2_+\leq 2\left( (u^*(z)-1)^2_++\{\int^z_0\partial_zu_{+}^* d\bar z\}^2\right).
\ee
Now take the average in $z$ on $[0,\sqrt{\epsilon_1}]$ to get
\begin{align}
(u-1)^2_+&\leq \frac{2}{\sqrt{\epsilon_1}}\left(\int^{\sqrt{\epsilon_1}}_0(u^*(z)-1)^2_+dz+\int^{\sqrt{\epsilon_1}}_0\{\int^z_0\partial_z u_{+}^* d\bar z\}^2dz)\right)\\
&\leq \frac{2}{\sqrt{\epsilon_1}} \left( \int^{\sqrt{\epsilon_1}}_0(u^*(z)-1)^2_+dz+\sqrt{\epsilon_1}\{\int^{\sqrt{\epsilon_1}}_0|\partial_zu_{+}^*| d\bar z\}^2\right)\\
&\leq \frac{2}{\sqrt{\epsilon_1}} \int^1_0(u^*(z)-1)^2_+dz+2\sqrt{\epsilon_1}\int^{1}_0|\nabla u_{+}^*|^2 dz.
\end{align}
Therefore on $Q_1$ we obtain
\begin{align}
\int_{Q_1}(u-1)^2_+dxds\leq \frac{2}{\sqrt{\epsilon_1}}\int_{Q_1^*}(u^*(z)-1)^2_+dxdzds+2\sqrt{\epsilon_1}\int_{Q^*_1}|\nabla u_{+}^*|^2 dxdzds\\
\leq (1+\sqrt{C_1})\sqrt{\epsilon_1}+2\sqrt{\epsilon_1}C_1, \quad\mbox{by}\quad \eqref{makeupsomething}, \eqref{makeupsomething2}.
\end{align}
Then we let
\be\label{Dthetadef}
D_\theta=1+\sqrt{C_1}+2C_1
\ee
and we would be finished.  Now, to establish $|U|\leq \epsilon_1$ note
\be\label{ml1}
|U|=\int^0_{-1}|\mathcal B(t)| dt,
\ee
where $\mathcal B(t)=\{(x,z)\in B^*_1: u^*(t,x,z)\geq 1\}$.  Define 
\be\label{mdefK}
K=\frac{4|B_1^*|\int^0_{-4}\int_{B_1^*} |\nabla u^\ast_+|^2dxdzdt}{\epsilon_1}.
\ee
We write
\[
[-4,0]=I \cup I^c,
\]
with 
\[
I=\big\{ t\in [-4,0]:\quad |\mathcal C(t)|^\frac{1}{2}\leq \epsilon_1^3\quad\mbox{and}\quad\int_{B^\ast_{1}}
|\nabla u^*_+|^2dxdz \leq K\big\}.
\]
If we can show $|I^c| \leq \frac{\epsilon_1}{2|B^*_1|}$ then \eqref{ml1} would become
\be\label{boundsonU}
\begin{split}
|U|=\int^0_{-1}|\mathcal B(t)| dt&= \int_{I^c\cap[-1,0]}|\mathcal B(t)| dt + \int_{I\cap[-1,0]}|\mathcal B(t)| dt\\
&\leq\frac{\epsilon_1}{2} + \sup_t|\mathcal B(t)|.
\end{split}
\ee
To estimate $|\mathcal B(t)|$ on $I\cap[-1,0]$ we use the De-Giorgi's Isoperimetric Lemma to obtain
\be\label{boundsonB}
|\mathcal B(t)|\leq \frac{\mathcal C(t)^\frac{1}{2}K^\frac{1}{2}}{|\mathcal A(t)|}.
\ee
If we could show $|\mathcal A(t)|\geq \frac{1}{4}$ on $I \cap [-1,0]$, \eqref{boundsonB} would imply
\[
 |\mathcal B(t)|\leq \sqrt{C_1}\frac{\epsilon_1}{2}\quad\mbox{on}\quad I\cap[-1,0],
\]
for $\epsilon_1$ small enough.
Therefore \eqref{boundsonU} would give us $|U|\leq \frac{\epsilon_1}{2}(1+\sqrt{C_1})$.  What is left to show is that $|I^c| \leq \frac{\epsilon_1}{2| B^*_1|}$ and that on $I\cap[-1,0]$, $|\mathcal A(t)|\geq \frac{1}{4}$. 
Start with the former and write $I^c=I^c_1 \cup I^c_2$ where
\[
I^c_1=\{t\in [-4,0]: \int_{B^\ast_{1}}|\nabla u^*_+|^2dxdz > K\} \quad\mbox{and}\quad  I^c_2=\{ t\in [-4,0]: |\mathcal C(t)|^\frac{1}{2}> \epsilon_1^3\}. 
\]
First
\begin{align*}
K|I^c_1|=\int_{I^c_1}Kdt&<\int_{I^c_1}\int_{B^\ast_{1}}|\nabla u^*_+|^2dxdzdt\\
&\leq\int^0_{-4}\int_{B^\ast_{1}}|\nabla u^*_+|^2dxdzdt\\
&=\frac{K}{4|B^*_1|}\epsilon_1\quad\mbox{by}\quad \eqref{mdefK}.
\end{align*}
Now set
\[
 \delta_1=\frac{\epsilon_1^8}{|B^*_1|}.
\]
Then for $I^c_2$ we have
\begin{align*}
|I^c_2|&<\frac{\int^0_{-4} |\mathcal C(t)| dt}{\epsilon_1^6}=\frac{|\{(t,x,z) \in Q^*_4 : 0 <u^*<1\}|}{\epsilon_1^6}
       <\frac{\delta_1}{\epsilon_1^6}
       &=\frac{\epsilon_1^2}{|B^*_1|}
       &\leq \frac{\epsilon_1}{4|B^*_1|}\quad\mbox{if $\epsilon_1$ small enough}.
\end{align*}
Hence $|I^c| \leq \frac{\epsilon_1}{2|B^*_1|}$ as needed.  Next we show $|\mathcal A(t)| \geq \frac{1}{4}$ for $t \in I\cap[-1,0]$.  We construct a sequence satisfying
\[
0\geq t_n\geq t_0+n\frac{\delta^*}{2}
\]
such that 
\be\label{somethingaboutA}
|\mathcal A(t)|\geq \frac{1}{4}\quad\mbox{on}\quad [t_n,t_n +\delta^*]\cap I \supset [t_n,t_{n+1}]\cap I.
\ee
We continue till $t_n +\delta^* \geq 1$, because then we can conclude $|\mathcal A(t)|\geq \frac{1}{4}$.  To pick the first element of the sequence we use the hypothesis that
\[
|\{(t,x,z)\in Q^*_1: u^*(t,x,z)\leq 0\}|\geq \frac{|Q^*_1|}{2}.
\]
Moreover, since $|I^c| \leq \frac{\epsilon_1}{2|B^*_1|}$ we can find some $t_0\in I \cap [-4,-1]$ such that $|\mathcal A(t_0)| \geq \frac{1}{4}$.  Next we would like to show
\be\label{1over64}
\int_{B_1}u^2_+(t)dx \leq \frac{1}{64} 
\ee
for every $t\in I$ and $t\geq t_0$ sufficiently close to $t_0$.  First, we consider $(u^\ast_+)^2(t_0)$ on $B^\ast_1$
\be\label{mepsilon2}
\begin{split}
\int_{B_1^*}(u^\ast_+)^2(t_0)dxdz &= \int_{\mathcal B(t_0)\cup \mathcal C(t_0)}(u^\ast_+)^2(t_0)dxdz\\
&\leq 4(|\mathcal B(t_0)|+ |\mathcal C(t_0)|)\\
&\leq 4(\epsilon_1^2+ \epsilon_1^6)\quad\mbox{since}\quad t\in I\\
&\leq 8\epsilon_1^2.
\end{split}
\ee
Second, for any $t$ and $z$
\[
\int_{B_1}u^2(t)dx=\int_{B_1} {u^*_+}^2(t,x,z)dx-2\int^z_0\int_{B_1}{u^*_+}(t)\partial_zu^*dxd\bar z,
\]
which once integrated in $z$ on $[0,1]$ gives for $t=t_0$
\begin{align}
\int_{B_1}u^2(t_0)dx &\leq\int_{B^*_1} {u^*_+}^2(t_0,x,z)dxdz+2\|u^*_+(t_0)\|_{L^2(B^*_1)}\|\nabla u^*_+(t_0)\|_{L^2(B^*_1)}\\
&\leq8\epsilon_1^2+4\sqrt{2}K^\frac{1}{2}\epsilon_1\quad\mbox{by \eqref{mepsilon2} and}\quad t\in I\\
&\leq C\sqrt{\epsilon_1^2}\quad\mbox{by \eqref{mdefK}}.
\end{align}
Now choosing an $\eta$ so that $\eta\rvert_{B_1^*}\leq1$, $\eta\leq r$ outside of $B_1^*$ and such that $|\nabla\eta|\sim\frac{1}{r}$, (where $r$ is to be chosen shortly) from the local energy inequality we have 
\[
\int_{B_1}u^2(t)dx \leq\int_{B_1} {u^*_+}^2(t_0)dx+ Cr +\frac{C(t-t_0)}{r}\leq C\sqrt{\epsilon_1^2}+ Cr +\frac{C(t-t_0)}{r}.
\]
Let $r$ be chosen so that
\[
Cr+C\sqrt{\epsilon_1}\leq \frac{1}{128},
\]
then for $t-t_0\leq \delta^*=\frac{r}{128C}$ \eqref{1over64} follows.\newline\indent  
Next we use \eqref{1over64} to get some preliminary lower bounds on the measure of $\mathcal A(t)$.  To begin with for $t \in I, t-t_0\leq\delta^*$ and $z\leq \epsilon_1^2$ write
\[
u^\ast_+(t,x,z)=u_+(t,x) +\int^z_0\partial z u^*_+(t,x,\bar z) d\bar z\leq u_+(t,x) +(\epsilon_1^2\int^1_0|\partial z u^*_+|^2 d\bar z)^\frac{1}{2}.
\]
Hence
\[
{u^\ast_+}^2(t,x,z)\leq 2 (u_+^2(t,x) +\epsilon_1^2\int^1_0|\partial z u^*_+|^2 d\bar z),
\]
and
\begin{align*}
\int_{B_1}{u^\ast_+}^2(t,x,z)dx &\leq 2 (\int_{B_1}u_+^2(t,x)dx +\epsilon_1^2\int_{B^*_1}|\partial_z u^*_+|^2 d\bar zdx)\\
&\leq 2 (\frac{1}{64} +\epsilon_1^2K)\\
&\leq 2 (\frac{1}{64} +4C\epsilon_1)\\
&\leq \frac{1}{4}.
\end{align*}
Then by Chebyshev's inequality, for every fixed $z\leq \epsilon_1$ we have
\[
|\{x\in B_1: u^*_+(t,x,z)\geq 1\}|\leq \frac{1}{4},
\]
which we now integrate in $z$ on $[0,\epsilon^2_1]$ to obtain
\[
|\{z\leq\epsilon^2_1, x\in B_1: u^*_+(t,x,z)\geq 1\}|\leq \frac{\epsilon^2_1}{4}.
\]
By definition of $\mathcal A(t)$ and $\mathcal C(t)$ we have
\[
B_1 \times [0,\epsilon^2_1] \subset \mathcal A(t) \cup \{z\leq\epsilon^2_1, x\in B_1: u^*_+(t,x,z)\geq 1\}\cup \mathcal C(t).
\]
Therefore
\begin{align*}
|\mathcal A(t)|&\geq |B_1|\epsilon^2_1-|\{z\leq\epsilon^2_1, x\in B_1: u^*_+(t,x,z)\geq 1\}|-|\mathcal C(t)|\\
&\geq \epsilon^2_1-\frac{\epsilon^2_1}{4}-\epsilon^6_1\geq \frac{\epsilon^2_1}{2}.
\end{align*}
Using \eqref{boundsonB} we have
\[
|\mathcal B(t)|\leq \tilde C\sqrt{\epsilon_1},
\]
and
\begin{align}
|\mathcal A(t)| \geq 1-|\mathcal B(t)|-|\mathcal A(t)|,\\
\geq 1-2\sqrt{\epsilon_1}-\epsilon_1^6\leq \frac{1}{4}.
\end{align}
as needed.  To pick the next element of the sequence we look at the interval $[t_0+\frac{\delta^*}{2},t_0+\delta^*]$ and observe that there must be some $t_1$ in that interval so that $t_1 \in I$.  We automatically also have $|\mathcal A(t_1)|\geq \frac{1}{4}$, so we can repeat the argument.  Again, we continue till $t_n +\delta^* \geq 1$ and since at each step  \eqref{somethingaboutA} holds, we can conclude $|\mathcal A(t)|\geq \frac{1}{4}$ as needed.

\section{Higher Regularity: Proof of Corollary \ref{cor1}.}\label{higher}
Extending H\"older continuity to higher regularity is not very difficult.  Indeed what is done in \cite{Driftdiffusion} in Appendix B can be applied here as well.  Therefore we only show how to set up the proof.  However, since some technical details are omitted in \cite{Driftdiffusion} for showing the solution is $C^\alpha$ for all $\alpha<1$, we illuminate them here.\newline\indent
Let $\theta : (0,\infty )\times \mathbb{R}^{N} \rightarrow \mathbb{R}$ be a solution
of the $N$-dimensional critical Burgers' equation which is essentially bounded and
locally Holder's continuous on $[\tau , \infty )\times \mathbb{R}^{N}$. That is,
\[
\theta \in L^{\infty}([\tau , \infty )\times \mathbb{R}^{N}) \cap C^{\alpha}_{loc}
([\tau , \infty )\times \mathbb{R}^{N})\]
for some $0 < \alpha < 1$.

Fix $y_{0} = (t_{0} , x_{0}) \in (\tau , \infty )\times \mathbb{R}^{N}$.  Without the loss of generality, we may assume\footnote{Otherwise replace the solution by $w(t,x) = \theta
(t, x + u(t-t_{0})) - \theta (y_{0})$, where $u$ is a vector in $\R^N$ with each entry equal to $\theta (y_{0})$.}
\[
\theta (y_{0}) = 0.
\]
Consider now the Poisson kernel
\[
P(t,x) = C_{N} \frac{t}{(t^{2} +
|x|^{2})^{\frac{N+1}{2}}},
\]
which is the fundamental solution for
\[
\partial_{t} + (-\triangle )^{\frac{1}{2}}=0.
\]
If $\theta_{0}= \theta (0,\cdot )$ is the initial datum for our solution $\theta$, using Duhamel's principle we have
\begin{equation*}
\begin{split}
\theta(t,x)
&=P(t,\cdot )* \theta_{0}(x) - \int_{0}^{t} P(t - t_{1} ,\cdot )* \sum_{j=1}^{N} (\theta \cdot
\partial_{j}\theta )(t_{1} ,\cdot)(x) dt_{1}\\
& =P(t,\cdot )* \theta_{0}(x) -\frac{1}{2} \sum_{j=1}^{N}  \int_{0}^{t}\int_{\mathbb{R}^{N}} \partial_{j}P(t -
t_{1} ,x - x_{1} )
\theta^{2}(t_{1}, x_{1}) dx_{1}dt_{1} .
\end{split}
\end{equation*}
Since $P(t,\cdot )* \theta_{0}$ is known to be $C^{\infty }$, we just need to examine
\[
 g(t,x)=-\frac{1}{2} \sum_{j=1}^{N}  \int_{0}^{t}\int_{\mathbb{R}^{N}} \partial_{j}P(t -
t_{1} ,x - x_{1} )
\theta^{2}(t_{1}, x_{1}) dx_{1}dt_{1}.
\]
For convenience, we now extend $P(t,x)$ to the whole space-time
$\mathbb{R} \times \mathbb{R}^{N}$ by \emph{requiring that $P(t, \cdot ) = 0$,
whenever $t < 0$}. With such an extension
\begin{equation*}
\begin{split}
g(t,x) &=
  -\frac{1}{2} \sum_{j=1}^{N}  \int_{-\infty}^{\infty}\int_{\mathbb{R}^{N}} 
\chi_{\{0 \leqslant t-t_{1}\leqslant t\}} \partial_{j}P(t - t_{1} ,x - x_{1} )
\theta^{2}(t_{1}, x_{1}) dx_{1}dt_{1}\\
& = -\frac{1}{2} \sum_{j=1}^{N}  \int_{\mathbb{R}^{N+1}} 
\chi_{\{t_{1}\geqslant 0\}} \partial_{j}P(t - t_{1} ,x - x_{1} )
\theta^{2}(t_{1}, x_{1}) dx_{1}dt_{1} .
\end{split}
\end{equation*}

 Our first task is to estimate the difference  $g(y_{0} + h e) -g(y_{0})$,
where $e\in \mathbb{R}^{N+1}$ with $|e| = 1$, and $h$ is some sufficiently small
positive number. We observe that
\begin{equation*}
g(y_{0} + h e) - g(y_{0}) =  -\frac{1}{2} \sum_{j=1}^{N}  \int_{\mathbb{R}^{N+1}} 
\chi_{\{t_{1}\geqslant 0\}} \{\partial_{j}P(y_{0} + h e  -y_{1}) -
\partial_{j}P(y_{0} - y_{1})\}
\theta^{2}(y_{1}) dy_{1} .
\end{equation*}
Therefore, we need to estimate the following two terms for each
$1\leqslant j \leqslant N$
\begin{align*}
A_{1} &=
\int_{B(y_{0},10h)} 
\chi_{\{t_{1}\geqslant 0\}} \partial_{j}P(y_{0} + h e  -y_{1}) \theta^{2}(y_{1})
dy_{1} - \int_{B(y_{0},10h)} \chi_{\{t_{1}\geqslant 0\}} \partial_{j}P(y_{0} -y_{1})
\theta^{2}(y_{1}) dy_{1},\\
A_{2} &= \int_{\mathbb{R}^{N+1}-B(y_{0},10h)} 
\chi_{\{t_{1}\geqslant 0\}} \{\partial_{j}P(y_{0} + h e  -y_{1}) -
\partial_{j}P(y_{0} - y_{1})\}
\theta^{2}(y_{1}) dy_{1},
\end{align*}
where $B(y_{0}, 10h) = \{ y_{1} \in \mathbb{R}^{N+1} : |y_{1} - y_{0}| < 10h  \}$.
Start with $A_1$.  Since 
\[
|\theta (y_{1})| = |\theta (y_{1}) - \theta (y_{0})| \leqslant C |y_{1} -
y_{0}|^{\alpha},
\]
and 
\[
|\nabla_{x} P (t,x)| \leqslant \frac{C_{N}(N+1)}{2} \frac{1}{(t^{2} +
|x|^{2})^{\frac{N+1}{2}}} = \frac{C_{N}(N+1)}{2y^{N+1}},
\]
we recognize that the second integral in the expression for $A_{1}$ can be controlled by 
\begin{equation}\label{Higherregularityone}
\begin{split}
| \int_{B(y_{0},10h)} \chi_{\{t_{1}\geqslant 0\}} \partial_{j}P(y_{0} -y_{1})
\theta^{2}(y_{1}) dy_{1}   | &\leqslant
 C \int_{B(y_{0}, 10h)}  \frac{1}{|y_{1} - y_{0}|^{N+1-2\alpha }} dy_{1}\\
& = C \int_{0}^{10h}\frac{r^{(N+1)-1}}{r^{N+1-2\alpha }} dr = C(N,\alpha) h^{2\alpha} .
\end{split}
\end{equation}
Next, to control the first term in the expression for $A_{1}$, we need the
following observation
\begin{itemize}
\item For every $t > 0$, $\partial_{j}P (t,\cdot )$ is an odd function in the
$x$-variable. Hence, the average value of $\partial_{j}P (t,\cdot )$ over any disc
$\{t\}\times \{ x\in \mathbb{R}^{N} : |x| < R \}$ centered at the $t$-axis must be
zero.
\end{itemize}
By the virtue of the above observation, it is easy to see 
\begin{equation*}
\int_{B(y_{0} + h e , 10h )} \chi_{\{t_{1} \geqslant 0\}} \partial_{j}P (y_{0} + h e
- y_{1}) dy_{1} = 0 . 
\end{equation*}
Then observe we can write
\begin{equation*}
\int_{B(y_{0}  , 10h )} \chi_{\{t_{1} \geqslant 0\}} \partial_{j}P (y_{0} + h e -
y_{1}) (\theta (y_{1}))^{2} dy_{1} = A_{11} + A_{12} - A_{13} ,  
\end{equation*}
where
\begin{itemize}
\item $A_{11} =  \int_{B(y_{0},10h)- B(y_{0} + h e , 10h)} 
\chi_{\{t_{1}\geqslant 0\}} \partial_{j}P(y_{0} + h e  -y_{1}) (\theta (y_{1}))^{2}
dy_{1}$ .
\item $ A_{12} =   \int_{B(y_{0},10h)\cap B(y_{0} + h e , 10h)} 
\chi_{\{t_{1}\geqslant 0\}} \partial_{j}P(y_{0} + h e  -y_{1}) \{(\theta
(y_{1}))^{2} - (\theta (y_{0} + he))^{2}  \} dy_{1} $.
\item $A_{13}  = \int_{B(y_{0} + h e , 10h)- B(y_{0} , 10h)}  \chi_{\{t_{1}\geqslant
0\}} \partial_{j}P(y_{0} + h e  -y_{1}) ( \theta (y_{0} + h e))^{2} dy_{1}   $.
\end{itemize}
We first look at $A_{12}$. By the Holder's continuity of $\theta$, we have 
\begin{equation*}
\begin{split}
| (\theta (y_{0} + h e))^{2}  - (\theta (y_{1}))^{2}    |
& \leqslant |\theta (y_{0} + h e) \{ \theta (y_{0} + h e) -\theta (y_{1}) \}|+ |\theta (y_{1}) \{ \theta (y_{0} + h e ) -\theta (y_{1})     \}|  \\
& \leqslant C (h^{\alpha } + |y_{1} - y_{0}|^{\alpha} ) |y_{0} + h e -
y_{1}|^{\alpha} .
\end{split}
\end{equation*}
If we further use that $y_{1} \in B(y_{0},10h)\cap B(y_{0} + h e , 10h)
$, the above inequality tells us
\begin{equation*}
| (\theta (y_{0} + h e))^{2}  - (\theta (y_{1}))^{2}| \leqslant C \{1 +
10^{\alpha}\} h^{\alpha} |y_{0} + h e - y_{1}|^{\alpha} . 
\end{equation*}
Thus
\begin{equation}\label{Higherregularity2}
|A_{12}| \leqslant C \int_{B( y_{0} + h e, 10h )} \frac{h^{\alpha}}{|y_{0} + h e -
y_{1}|^{N+1-\alpha} } dy_{1} = C h^{2\alpha } . 
\end{equation}
On the other hand, the terms $A_{11}$ and $A_{13}$ can be handled in the following way
\begin{equation}\label{Higherregularity3}
\begin{split}
|A_{11}| 
& \leqslant C \int_{B(y_{0},10h)- B(y_{0} + h e , 10h)} |\partial_{j}P (y_{0} + h e
-y_{1})|\cdot |y_{1}-y_{0}|^{2\alpha} dy_{1}\\
& \leqslant (10h)^{2\alpha} C \int_{B(y_{0},10h)- B(y_{0} + h e , 10h)} \frac{1}{
|y_{0} + h e - y_{1}|^{N+1} } dy_{1} \\
& \leqslant C h^{2\alpha} \int_{\{  10h \leqslant |y_{0} + h e - y_{1}| \leqslant 11
h \}}\frac{1}{ |y_{0} + h e - y_{1}|^{N+1} } dy_{1} \\
& = C h^{2\alpha} \log(\frac{11}{10}),
\end{split}
\end{equation}
and 
\begin{equation}\label{Higherregularity4}
\begin{split}
|A_{13}|
& \leqslant  \int_{B(y_{0} + h e , 10h)- B(y_{0} , 10h)}  |\partial_{j}P(y_{0} + h e
 -y_{1})| ( \theta (y_{0} + h e))^{2} dy_{1} \\
& \leqslant C h^{2\alpha} \int_{\{  9h \leqslant |y_{0} + h e - y_{1}| \leqslant 10
h \}} \frac{1}{ |y_{0} + h e - y_{1}|^{N+1} } dy_{1} \\
&= C h^{2\alpha} \log(\frac{10}{9}) .
\end{split}
\end{equation}
By combining \eqref{Higherregularity2}, \eqref{Higherregularity3}, \eqref{Higherregularity4}, we can conclude 
\[
\int_{B(y_{0}  , 10h )} \chi_{\{t_{1}
\geqslant 0\}} \partial_{j}P (y_{0} + h e - y_{1}) (\theta (y_{1}))^{2} dy_{1}\leq C h^{2\alpha},    
\]
which with \eqref{Higherregularityone} implies
\[
|A_{1}| \leqslant C h^{2\alpha}.
\]
To complete the estimate for $|g(y_{0}+ h e) - g(y_{0})|$, we also
need to control $|A_{2}|$.  For this purpose, we first recall the derivatives of
$\partial_{j}P$
\begin{itemize}
\item $\partial_{i} \partial_{j}P (t,x) = -C_{N}(N+1)t\{
\frac{\delta_{ij}}{(t^{2}+|x|^{2})^{\frac{N+3}{2}}}  - \frac{(N+3)x_{i}x_{j}}{
(t^{2}+|x|^{2})^{\frac{N+3}{2} +1}            }  \} $ .
\item $\partial_{t} \partial_{j} P (t,x) = -C_{N}(N+1)x_{j}  
 \{ \frac{1}{(t^{2}+|x|^{2})^{\frac{N+3}{2}}}  - \frac{(N+3)t^{2}}{
(t^{2}+|x|^{2})^{\frac{N+3}{2} +1}            }  \}  $
\end{itemize}
Then, it follows directly from the above two identities that
\[
|\nabla \partial_j P (y)| \leqslant \frac{C}{|y|^{N+2}},\quad \forall y = (t,x) \in
\mathbb{R}^{N+1}-{(0,0)}.
\]
Therefore for all $y_{1}
\in \mathbb{R}^{N+1} - B(y_{0} , 10h)$
\begin{equation*}
\begin{split}
|\partial_{j}P (y_{0} + h e - y_{1})  - \partial_{j}P (y_{0} - y_{1})|
& \leqslant \int_{0}^{h} |\nabla \partial_{j}P (y_{0} +\lambda e - y_{1})| d\lambda \\
& \leqslant \int_{0}^{h} \frac{C}{| y_{0} +\lambda e - y_{1}    |^{N+2}} d\lambda \\
& \leqslant \int_{0}^{h} \frac{C}{ \frac{9}{10} |y_{1} -y_{0}|^{N+2}} d\lambda = C \frac{h}{|y_{1} - y_{0}|^{N+2}} ,
\end{split}
\end{equation*}
where use the fact that for any $y_{1} \in
\mathbb{R}^{N+1} - B(y_{0} , 10h)$, we have $|y_{1} + \lambda e - y_{0}|\geqslant
|y_{1} - y_{0}|-h \geqslant \frac{9}{10} |y_{1} - y_{0}|$. As a result, we can
control $|A_{2}|$ as follows if we have $2\alpha <1$
\begin{equation*}
\begin{split}
|A_{2}|
&\leqslant \int_{\{  |y_{1} -y_{0}| \geqslant 10h  \}}  |\partial_{j}P (y_{0} + h e
- y_{1})  - \partial_{j}P (y_{0} - y_{1})| (\theta (y_{1}))^{2} dy_{1} \\
& \leqslant \int_{\{  |y_{1} -y_{0}| \geqslant 10h  \}} \frac{Ch}{|y_{1} -
y_{0}|^{N+2-2\alpha } } dy_{1}\\
& = C h \int_{10h}^{\infty} r^{-2+2\alpha } dr = h^{2\alpha} .
\end{split}
\end{equation*}
So we conclude that if $\alpha$ satisfies $2\alpha < 1$, then, we must have
\[ 
|g(y_{0} + h e) - g(y_{0})| \leqslant C h^{2\alpha },
\]
for all sufficiently small $h > 0$. This means that $\theta$ must be of class $C^{2\alpha}$ also, provided
$\theta$ is of class $C^{\alpha}$.  By bootstrapping the above argument, we may now
conclude that our locally Holder's continuous function $\theta$ is of class
$C^{\gamma}$, for any $0 < \gamma <1$.\newline\indent
To go beyond Lipschitz and obtain the $C^{1,\beta}$ regularity for $\theta$, we just
need to follow the argument in the second part of Appendix B \cite{Driftdiffusion}. 
\bibliography{chihinsbib}
\bibliographystyle{plain}
\vspace{.125in}

\end{document}